\documentclass[11pt]{amsart}

\usepackage[hang,flushmargin]{footmisc}

\usepackage{amscd,amssymb,amsopn,amsmath,amsthm,mathrsfs,graphics,amsfonts,enumerate,verbatim,calc
}
\usepackage{bbm}
\usepackage[all,cmtip]{xy}
\usepackage[all]{xy}
\usepackage{tikz}
\usepackage{lscape}
\usepackage{enumitem}
\usetikzlibrary{matrix,arrows,decorations.pathmorphing,cd}
\usepackage{hyperref}
\hypersetup{colorlinks=true,linkcolor={blue}}
\usepackage{cleveref}
\usepackage{url}

\usepackage{adjustbox}

\usepackage{scalerel}
\usepackage{stackengine,wasysym}
\usetikzlibrary {positioning}
\usetikzlibrary{patterns}
\usetikzlibrary{calc}
\definecolor {processblue}{cmyk}{0.96,0,0,0}

\usepackage{subfigure}

\usepackage{comment}

\usepackage{ dsfont }

\usepackage{color}

\usepackage[OT2,OT1]{fontenc}
\newcommand\cyr{%
\renewcommand\rmdefault{wncyr}
\renewcommand\sfdefault{wncyss}
\renewcommand\encodingdefault{OT2}
\normalfont
\selectfont}
\DeclareTextFontCommand{\textcyr}{\cyr}

\usepackage{amssymb,amsmath}

\DeclareFontFamily{OT1}{rsfs}{}
\DeclareFontShape{OT1}{rsfs}{n}{it}{<-> rsfs10}{}
\DeclareMathAlphabet{\mathscr}{OT1}{rsfs}{n}{it}

\numberwithin{equation}{section}
\newtheorem{theorem}{Theorem}[section]
\newtheorem{lem}[theorem]{Lemma}
\newtheorem{cor}[theorem]{Corollary}

\newtheorem{quest}[theorem]{Question}

\newtheorem{prop}[theorem]{Proposition}

\theoremstyle{definition}
\newtheorem{defn}[theorem]{Definition}

\theoremstyle{remark}
\newtheorem{remark}[theorem]{Remark}

\newtheorem{example}[theorem]{Example}

\renewcommand{\tilde}{\widetilde}

\def\tr{(\mathbf{tr})}
\def\GC{(\mathbf{GC})}
\def\dep{(\mathbf{dep})}
\def\ldep{(\mathbf{ldep})}
\def\rdep{(\mathbf{rdep})}
\def\uac{(\mathbf{UAC})}
\def\ubc{(\mathbf{UBC})}

\newcommand{\hit}{\operatorname{ht}}

\newcommand{\Ass}{\operatorname{Ass}}

\renewcommand{\ker}{\operatorname{ker}}

\newcommand{\Spec}{\operatorname{Spec}}

\newcommand{\NF}{\operatorname{NF}}

\newcommand{\pd}{\operatorname{pd}}
\newcommand{\id}{\operatorname{id}}
\newcommand{\Ext}{\operatorname{Ext}}
\newcommand{\cone}{\operatorname{cone}}
\newcommand{\Supp}{\operatorname{Supp}}
\newcommand{\Tor}{\operatorname{Tor}}
\newcommand{\Hom}{\operatorname{Hom}}

\newcommand{\depth}{\operatorname{depth}}

\newcommand{\codepth}{\operatorname{codepth}}

\newcommand{\coker}{\operatorname{coker}}

\newcommand{\Min}{\operatorname{Min}}

\newcommand{\rank}{\ensuremath{\operatorname{rank}}}

\newcommand{\Tr}{\operatorname{Tr}}

\newcommand{\p}{\mathfrak{p}}
\newcommand{\q}{\mathfrak{q}}
\newcommand{\m}{\mathfrak{m}}

\newcommand{\Z}{\mathbb{Z}}

\newcommand{\w}{\omega}

\renewcommand{\bar}{\overline}

\newcommand{\Ann}{\operatorname{Ann}}

\newcommand{\RHom}{\operatorname{\mathbf{R}Hom}}

\begin{document}
\title[On the depth of tensor products over Cohen-Macaulay rings]{On the depth of tensor products over Cohen-Macaulay rings}

\author[Kimura]{Kaito Kimura}
\address[Kimura]{Graduate School of Mathematics, Nagoya University, Furocho, Chikusaku, Nagoya 464-8602, Japan}
\email{m21018b@math.nagoya-u.ac.jp}
\urladdr{https://sites.google.com/view/kaitokimura/}

\author[Lyle]{Justin Lyle}
\address[Lyle]{Department of Mathematics and Statistics \\ 221 Parker Hall\\
Auburn University\\
Auburn, AL 36849}
\email[Justin Lyle]{jll0107@auburn.edu}
\urladdr{https://jlyle42.github.io/justinlyle/}

\author[Soto Levins]{Andrew J. Soto Levins}
\address[Soto Levins]{Department of Mathematics and Statistics \\
Texas Tech University \\
Lubbock, TX 79409 USA}
\email[Andrew J. Soto Levins]{ansotole@ttu.edu}
\urladdr{https://sites.google.com/view/andrewjsotolevins}



\begin{abstract}
Inspired by classical work on the depth formula for tensor products of finitely generated $R$-modules, we introduce two conditions which we call $\ldep$ and $\rdep$ and their derived variations. We show for Cohen-Macaulay local rings that derived $\ldep$ is equivalent to $\dim(R)$ being a uniform Auslander bound for $R$, and if $\dim(R)>0$ that both are equivalent to $\ldep$. We introduce an analogous condition we call the \emph{uniform Buchweitz condition} and provide a corresponding theorem for the $\rdep$ condition. As a consequence of these results, we show $\ldep$ implies $\rdep$ when $R$ is Gorenstein and that the $\ldep$ and $\rdep$ conditions behave well under modding out by regular sequences and completion, but we give a concrete example showing they need not localize. Using our methods, we extend work of Jorgensen \cite{Jo99} by calculating the value $q_R(M,N):=\sup\{i \mid \Tor^R_i(M,N) \ne 0\}$ under certain conditions. 
    
\end{abstract}

\maketitle

\section{Introduction}
Let $(R,\m,k)$ be a (commutative) Noetherian local ring and let $M$ and $N$ be finitely generated $R$-modules. Understanding the depth of $M \otimes_R N$ has been an open line of investigation for several decades, and is an invariant intrinsically motivated by geometric phenomena. In particular, when $M$ and $N$ are $\Tor$-independent, that is when $\Tor^R_i(M,N)=0$ for all $i>0$, we philosophize that the homological independence of $M$ and $N$ should lead $\depth_R(M \otimes_R N)$ to be computable in a predictable way. This intuition can made precise when both $M$ and $N$ have finite projective dimension where it is relatively easy to see that the condition $\Tor^R_{i>0}(M,N)=0$ forces $\pd_R(M \otimes_R N)=\pd_R(M)+\pd_R(N)$. Equivalently, setting $\codepth_R(-):=\depth(R)-\depth_R(-)$, we have 
\[\codepth_R(M \otimes_R N)=\codepth_R(M)+\codepth_R(N)\] by the Auslander-Buchsbaum formula, which is commonly expressed in the literature as \[\depth_R(M \otimes_R N)+\depth(R)=\depth_R(M)+\depth_R(N)\] 
and simply referred to as the ``depth formula". 

It is natural to wonder in what generality the $\Tor$-independence condition $\Tor^R_{i>0}(M,N)=0$ will force the depth formula to hold. The seminal work of Auslander in 1961 shows that it is sufficient that only one of $M$ or $N$ have finite projective dimension \cite{Au61}. This result has been improved numerous times in the ensuing decades. It was shown to hold for any pair of $\Tor$-independent modules over a complete intersection in \cite{HW94}, and this was improved by Iyengar to show it is sufficient that the module $M$ have finite complete intersection dimension \cite{iyengar}. 

Christensen-Jorgensen showed that the depth formula holds for any pair of $\Tor$-independent modules over a so-called AB-ring, which is a Gorenstein ring satisfying the so-called \emph{Uniform Auslander Condition $\uac$} (see Definition \ref{uacdef}), and they show AB-rings even satisfy a derived variation. Recent work of Kimura-Lyle-Otake-Takahashi establishes the converse in the sense that the depth formula holds for a Gorenstein local ring $R$ of positive dimension if and only if $R$ is AB, and they use this characterization to give the first examples of local rings $R$ and $\Tor$-independent $R$-modules $M$ and $N$ for which the depth formula does not hold \cite{KL23}. 

In this work, we explore the depth formula in absence of the Gorenstein hypothesis and reveal the surprising subtleties that emerge when this condition is removed. In particular, in this setting it is more natural to bifurcate the depth formula into two conditions we call $\ldep$ and $\rdep$ based on each possible inequality (see Definition \ref{depdef}). In considering these conditions and their derived variants (see Definitions \ref{deriveddepdef} and \ref{deriveddepdefformodules}), we demonstrate that it is really the $\ldep$ condition that is tightly connected to the $\uac$. 

To present our results, we set the notation $q_R(M,N):=\sup\{i \mid \Tor^R_i(M,N) \ne 0\}$ and $b_R(M,N):=\sup\{i \mid \Ext^i_R(M,N) \ne 0\}$. Concretely, we prove the following (see Theorem \ref{ldepthm}):
\begin{theorem}\label{ldepintro}
Let $R$ be a CM ring of dimension $d$. Consider the following conditions:
\begin{enumerate}
\item[$(1)$] $R$ satisfies derived $\ldep$. 
\item[$(2)$] $R$ satisfies derived $\ldep$ for modules. 
\item[$(3)$] For all finite length modules $A,B$, and setting $M:=\Omega^d_R(A)$ and $N:=\Omega^d_R(B)$, if $q_R(M,N)<\infty$, then $M \otimes^L_R N$ is an MCM complex (in the sense of \cite{IM21}).
\item[$(4)$] If $M,N$ are MCM $R$-modules with $q_R(M,N)<\infty$, then $q_R(M,N)=0$ and $M \otimes_R N$ is MCM.
\item[$(5)$] If $M$ and $N$ are $R$-modules with $b_R(M,N)<\infty$, then $b_R(M,N) \le \codepth_R(M)$. 
\item[$(6)$] $R$ satisfies $\uac$ with $b_R=d$.
\item[$(7)$] For MCM $R$-modules $M$ and $N$, if $b_R(M,N)<\infty$, then $b_R(M,N)=0$.
\item[$(8)$] $R$ satisfies $\ldep$.
    
\end{enumerate}
Then we have the following:
\begin{enumerate}[label=(\upshape{\Roman*})]
\item Conditions $(1)-(6)$ are equivalent.
\item Conditions $(1)-(6)$ imply $(7)$ and the converse holds if $R$ admits a canonical module $\w_R$. 
\item Conditions $(1)-(6)$ imply $(8)$ and the converse holds if $\dim(R)>0$. 
    
\end{enumerate}
    
\end{theorem}

Combining Theorem \ref{ldepintro} with work of Christensen-Jorgensen \cite[Corollary 5.3]{CJ15}, we show that $\ldep$ implies $\rdep$ for Gorenstein local rings of positive dimension (see Corollary \ref{gorcase}), thus recovering and refining \cite[Theorem 3.6]{KL23}.

By contrast, the $\rdep$ possesses certain features which distinguish it from $\ldep$ and is overall much harder to control. This is surprising in a sense, since one generally expects tensor products to have bad depth properties outside of special circumstances. For instance, while Theorem \ref{ldepintro} shows that derived $\ldep$ for modules is equivalent to the relatively strong $\uac$ condition even for Artinian algebras, we show that every Artinian local ring satisfies derived $\rdep$ for modules (Example \ref{derivedrdepArtinian}); one must allow consideration of complexes rather than modules to get a meaningful version of $\rdep$ in this setting. Nonetheless, we introduce the \emph{uniform Buchweitz condition} $\ubc$ to serve as a dual notion of $\uac$ (see Definition \ref{ubcdef} for the definition) and provide the analogue of Theorem \ref{ldepintro} for $\rdep$ with the following (see Theorem \ref{rdepthm}):
\begin{theorem}\label{rdepintro}
Let $R$ be a CM ring of dimension $d$. Consider the following conditions:
\begin{enumerate}
\item[$(1)$] $R$ satisfies derived $\rdep$.
\item[$(2)$] $R$ satisfies derived $\rdep$ for modules.
\item[$(3)$] $R$ satisfies $\rdep$.
\item[$(4)$] If $M$ and $N$ are $R$-modules with $q_R(M,N)=0$ that are locally free of constant rank on the punctured spectrum of $R$, then $\codepth_R(M \otimes_R N) \ge \codepth_R(M)+\codepth_R(N)$.
\item[$(5)$] If $M$ and $N$ are $R$-modules with $q_R(M,N)<\infty$, and if $M \otimes_R^L N$ is an MCM complex, then $M$ is MCM.
\item[$(6)$] If $M$ and $N$ are $R$-modules with $N$ MCM, with $q_R(M,N)=0$, and with $M \otimes_R N$ MCM, then $M$ is MCM.
\item[$(7)$] $R$ satisfies $\ubc$.
\item[$(8)$] If $M$ and $N$ are $R$-modules with $N$ MCM and $b_R(M,N)=0$, then $M$ is MCM.
\end{enumerate}

Then $(1) \Rightarrow (2) \iff (3) \iff (4) \Rightarrow (5) \Rightarrow (6)$, and $(3) \Rightarrow (7) \Rightarrow (8)$. If $d>0$, then $(2) \Rightarrow (1)$, and if $R$ has a canonical module $\w_R$, then $(6) \iff (8)$.
\end{theorem}

Theorems \ref{ldepintro} and \ref{rdepintro} notably demonstrate that phenomena such as the $\uac$, which is motivated from the representation theory of Artin algebras, can be detected even in the Artinian setting via depth conditions. The key point is that while the depth of a module is always zero and so does not contain meaningful information in the Artinian setting, depth for complexes behaves as robustly as it does in higher dimension.

A consequence of Theorems \ref{ldepintro} and \ref{rdepintro} is that $\ldep$ and $\rdep$ behave well under certain standard operations in commutative algebra e.g.~modding out by regular sequences and completion (see Propositions \ref{ldepcutsdown} and \ref{cutsdown}, and Corollaries \ref{ldepcompletes}, \ref{ldepcomplete}, and \ref{rdepcompletes}). However, we give a concrete example to show they need not localize (Example \ref{notlocalize}).

Inspired by work of Jorgensen \cite{Jo99}, we also describe the explicit connection between the $\ldep$ and $\rdep$ conditions and the invariant $q_R(M,N)$ when it is finite. Our main theorem in this direction gives a direct extension of \cite[Theorem 2.2]{Jo99} as the derived $\dep$ condition is known to hold for modules of finite complete intersection dimension, and since finiteness of complete intersection dimension localizes (see Remark \ref{jorrmk}).
\begin{theorem}\label{qrtheoremintro}
Suppose $R$ is a Noetherian local ring and let $M$ and $N$ be finitely generated $R$-modules with $q_R(M,N)<\infty$. Then we have the following:
\begin{enumerate}
\item[$(1)$] If $R_{\p}$ satisfies derived $\ldep$ for all $\p \in \Spec(R)$, then 
\[q_R(M,N) \le \sup\{\depth(R_{\p})-\depth_R(M_{\p})-\depth_R(N_{\p}) \mid \p \in \Supp(M) \cap \Supp(N)\}.\]
\item[$(2)$] If $R_{\p}$ satisfies derived $\rdep$ for all $\p \in \Spec(R)$, then
\[q_R(M,N) \ge \sup\{\depth(R_{\p})-\depth_R(M_{\p})-\depth_R(N_{\p}) \mid \p \in \Supp(M) \cap \Supp(N)\}.\]
\end{enumerate}
\end{theorem}

We now outline the structure of our paper: In Section \ref{background} we set notation and provide necessary background. In Section \ref{prelimsection}, we provide several preliminary results of independent interest that are instrumental in the proofs of our main theorems. We also include Subsection \ref{subsection} which compares derived $\ldep$/$\rdep$ with the corresponding versions for modules.  Section \ref{ldepsection} contains the proof of Theorem \ref{ldepintro} and Section \ref{rdepsection} is devoted to the proof of Theorem \ref{rdepintro}. In Section \ref{qrsection}, we prove Theorem \ref{qrtheoremintro}. Section \ref{questionsection} discusses several open questions brought forth by our work.

\section{Background and Notation}\label{background}
Throughout, we let $(R,\m,k)$ be a Noetherian local ring. All modules or complexes of modules are assumed to be finitely generated unless otherwise stated. For an $R$-complex $C$ we denote the $i$th graded component under the homological grading by $C_i$ and under the cohomological grading $C^i$, so $C^i=C_{-i}$. If $C$ is an $R$-complex, then for any $i \in \Z$, we let $\tau_{\ge i}(C):= \cdots \rightarrow C_{i+1} \rightarrow C_i \rightarrow 0$ and $\tau^{\ge i}(C):=0 \rightarrow C^i \rightarrow C^{i+1} \rightarrow \cdots$ denote the \emph{brutal truncations of $C$}. So there is a natural surjection of chain complexes $C \twoheadrightarrow \tau_{\ge i}(C)$ and a natural injection of chain complexes $\tau^{\ge i}(C) \hookrightarrow C$ for any $i$.

\begin{defn}\label{depthdef}
Let $X= \cdots X_{i+1} \to X_i \to X_{i-1} \to \cdots$ be a complex of $R$-modules. Let $I$ be an injective resolution of $X$, that is, a complex of injective $R$-modules that is quasi-isomorphic to $X$. 
We define $\Ext^i_R(k,X):=H_{-i}(\Hom_R(k,I))$ and $H^i_{\m}(X):=H_{-i}(\Gamma_{\m}(I))$ where $\Gamma_{\m}(-)$ denotes the $\m$-power torsion functor. Then the depth of $X$ is the value 
\[\depth_R(X):=\inf\{i \mid \Ext^i_R(k,X) \ne 0\}=\inf\{i \mid H_{\m}^i(X) \ne 0\}.\]
where the last equality follows from \cite[Theorem I]{FI03}. 
    
\end{defn}

\begin{remark}
The notion of depth introduced in Definition \ref{depthdef} can be defined in numerous ways analogous to the familiar notion of depth for modules (see \cite{FI03}), and it is clear that the depth of a complex is invariant under quasi-isomorphism; in particular, the depth of a module in the familiar sense agrees with its depth viewed as a complex concentrated in degree zero. The definition we have provided is the one most salient to our needs. In light of the discussion above, we will often think of $\depth$ as being defined on objects in the derived category. 
\end{remark}

The following definition was introduced by \cite{IM21} and used to provide new proofs of several of the classical homological conjectures:
\begin{defn}\label{mcmdef}
We say an $R$-complex $X$ is maximal Cohen-Macaulay (MCM) if the following hold:
\begin{enumerate}
\item[$(1)$] $H(X)$ is finitely generated.
\item[$(2)$] The natural map $H_0(X) \to H_0(k \otimes^L_R X)$ is nonzero.
\item[$(3)$] $H^i_{\m}(X)=0$ if $i \ne \dim(R)$.
\end{enumerate}
\end{defn}

\begin{defn}\label{uacdef}
If $M$ and $N$ are $R$-modules, we set $$q_R(M,N):=\sup\{i \mid \Tor^R_i(M,N) \ne 0\}$$ and we set $$b_R(M,N):=\sup\{i \mid \Ext^i_R(M,N) \ne 0\}.$$ We set $$b_R:=\sup_{A,B}\{b_R(A,B) \mid b_R(A,B)<\infty\}$$ and call it the \emph{Auslander bound} of $R$. We say $R$ satisfies the \emph{Uniform Auslander Condition} $\uac$ if $b_R<\infty$. A Gorenstein ring satisfying $\uac$ is said to be $AB$.
\end{defn}

We introduce the following definition that serves as a dual version of $\uac$:

\begin{defn}\label{ubcdef}
We say $R$ satisfies the \emph{Uniform Buchweitz Condition} $\ubc$ if for all (nonzero) $R$-modules $M$ and $N$ with $b_R(M,N)<\infty$, then $b_R(M,N) \ge \codepth_R(M)$.
\end{defn}

\begin{remark}
It is perhaps unclear what is uniform about the $\
ubc$ condition, especially given that $\uac$ prescribes the existence of the finite bound $b_R$, which in theory may vary. The key point is that while one may try to define $\ubc$ to mean the existence of a value $a_R$ for which $b_R(M,N)<\infty$ implies $b_R(M,N) \ge a_R$, grade obstructions preclude its uniformity. Indeed, if $M$ is any module of finite projective dimension then $b_R(M,R)=\pd_R(M)=\codepth_R(M)$. The condition we have provided is thus the best one can expect to hold uniformly. 
\end{remark}

The following definitions serve as the main objects of study in this work.

\begin{defn}\label{depdef}

Let $R$ be a local ring. 

\begin{enumerate}

\item[$(1)$] We say $R$ satisfies $\ldep$ if for any finitely-generated modules $M$ and $N$ with $q_R(M,N)=0$, we have $\depth_R(M \otimes_R N)+\depth(R) \ge \depth_R(M)+\depth_R(N)$. Equivalently, we have $\codepth(M \otimes_R N) \le \codepth(M)+\codepth(N)$.

\item[$(2)$] We say $R$ satisfies $\rdep$ if for any finitely-generated modules $M$ and $N$ with $q_R(M,N)=0$, we have $\depth_R(M \otimes_R N)+\depth(R) \le \depth_R(M)+\depth_R(N)$. Equivalently, we have $\codepth(M \otimes_R N) \ge \codepth(M)+\codepth(N)$.

\item[$(3)$] We say $R$ satisfies $\dep$ if it satisfies both $\ldep$ and $\rdep$, i.e., if for any finitely-generated modules $M$ and $N$ with $q_R(M,N)=0$ we have $\codepth(M \otimes_R N)=\codepth(M)+\codepth(N)$. 
\end{enumerate}
\end{defn}

These conditions each admit a corresponding derived version:
\begin{defn}\label{deriveddepdef}

Let $R$ be a local ring. 

\begin{enumerate}

\item[$(1)$] We say $R$ satisfies \emph{derived $\ldep$} if for any complexes of $R$-modules $M$ and $N$ for which $M$, $N$ and $M \otimes^L_R N$ have bounded homology, we have $\depth_R(M \otimes^L_R N)+\depth(R) \ge \depth_R(M)+\depth_R(N)$. 

\item[$(2)$] We say $R$ satisfies \emph{derived $\rdep$} if for any complexes of $R$-modules $M$ and $N$ for which $M$, $N$ and $M \otimes^L_R N$ have bounded homology, we have $\depth_R(M \otimes^L_R N)+\depth(R) \le \depth_R(M)+\depth_R(N)$. 

\item[$(3)$] We say $R$ satisfies \emph{derived $\dep$} if it satisfies both derived $\ldep$ and derived $\rdep$, i.e. if for any complexes of $R$-modules $M$ and $N$ for which $M$, $N$ and $M \otimes^L_R N$ have bounded homology we have $\depth_R(M \otimes^L_R N)+\depth(R)=\depth_R(M)+\depth_R(N)$.
\end{enumerate}
\end{defn}

The following special case of the derived $\ldep$ and derived $\rdep$ conditions is also instrumental in our analysis.

\begin{defn}\label{deriveddepdefformodules}

Let $R$ be a local ring. 

\begin{enumerate}

\item[$(1)$] We say $R$ satisfies \emph{derived $\ldep$ for modules} if for any finitely-generated modules $M$ and $N$ with $q_R(M,N)<\infty$, we have $\depth_R(M \otimes^L_R N)+\depth(R) \ge \depth_R(M)+\depth_R(N)$. Equivalently, we have $\codepth(M \otimes^L_R N) \le \codepth(M)+\codepth(N)$.

\item[$(2)$] We say $R$ satisfies \emph{derived $\rdep$ for modules} if for any finitely-generated modules $M$ and $N$ with $q_R(M,N)<\infty$, we have $\depth_R(M \otimes^L_R N)+\depth(R) \le \depth_R(M)+\depth_R(N)$. Equivalently, we have $\codepth(M \otimes^L_R N) \ge \codepth(M)+\codepth(N)$.

\item[$(3)$] We say $R$ satisfies \emph{derived $\dep$ for modules} if it satisfies both derived $\ldep$ and derived $\rdep$, i.e., if for any finitely-generated modules $M$ and $N$ with $q_R(M,N)<\infty$ we have $\codepth(M \otimes^L_R N)=\codepth(M)+\codepth(N)$. 
\end{enumerate}
\end{defn}

\begin{remark}\label{derivedimplies}
It is easy to see that derived $\ldep$ implies derived $\ldep$ for modules which implies $\ldep$ and that derived $\rdep$ implies derived $\rdep$ for modules which implies $\rdep$. A consequence of the main results of this paper is that in nearly all situations the differing $\ldep$ conditions and the differing $\rdep$ conditions will turn out to be equivalent; see Theorems \ref{ldepthm} and \ref{rdepthm}. 
\end{remark}

A key source of examples of rings satisfying the $\ldep$ and $\rdep$ conditions is given by the following:

\begin{example}\label{tvexample}
If $R$ satisfies the trivial vanishing condition of \cite{LM20} \footnote{\noindent Trivial vanishing was essentially considered in \cite{JS04}, but only the weaker variant for $\Ext$. It is also called $\Tor$-friendly in \cite{AI22}.}, that is, if $q_R(M,N)<\infty$ implies that one of $M$ or $N$ has finite projective dimension, then $R$ satsfies derived $\dep$. Indeed, since one of $M$ or $N$ must have finite projective dimension when $q_R(M,N)<\infty$ in this situation, the claim follows from \cite[Corollary 2.2]{iyengar}. Known classes of rings satisfying trivial vanishing include Golod rings and rings with certain small numerics (see \cite[Theorem 3.1]{Jo99} and \cite[Section 4]{LM20}). 

We note also that any ring satisfying trivial vanishing must satisfy $\ubc$. Indeed, if $R$ satisfies trivial vanishing, and if $b_R(M,N)<\infty$, then by \cite[Theorem 3.2]{LM20}, we have that either $M$ has finite projective dimension or that $N$ has finite injective dimension. If $M$ has finite projective dimension then it is well-known that $\Ext^{\pd_R(M)}_R(M,N) \ne 0$, while $\pd_R(M)=\codepth_R(M)$ by the Auslander-Buchsbaum formula. If instead $\id_R(N)<\infty$ then $\Ext^{\codepth_R(M)}_R(M,N) \ne 0$ by e.g. \cite[Theorem 11.2]{LW12}, so $R$ satisfies $\ubc$.

\end{example}

We finally recall the notion of Gorenstein dimension. We write $(-)^*:=\Hom_R(-,R)$.
\begin{defn}\label{gdimdef}
An $R$-module $M$ is said to be \emph{totally reflexive} if the following conditions hold:
\begin{enumerate}
\item[$(1)$] $M$ is reflexive, i.e. the natural biduality map $M \to M^{**}$ is an isomorphism.
\item[$(2)$] $b_R(M,R)=0$.
\item[$(3)$] $b_R(M^*,R)=0$.
\end{enumerate}
\end{defn}

A question of Yoshino \cite{Yo05} originally asked whether condition $(2)$ of Definition \ref{gdimdef} was sufficient to imply the other two, but examples of Jorgensen-\c{S}ega show this is not the case \cite{JS06}. In light of this, we say $R$ satisfies the $\tr$ condition if $b_R(M,R)=0$ implies $M$ is totally reflexive for any $R$-module $M$. \footnote{The $\tr$ condition is a key focus of \cite{KL23} and is also called the $\GC$ condition in \cite{CH10}.}

\section{Preliminary Results}\label{prelimsection}
In this section we provide several results that will ultimately be instrumental for the proofs of the main results in future sections. Our focus in this work is the case where $R$ is Cohen-Macaulay (CM) and we will make this assumption throughout. We fix a maximal regular sequence $\underline{x}$ and we set $\tilde{(-)}:=\Omega^d_R(M/\underline{x}M)$ for ease of notation. When $R$ has a canonical module $\w_R$, we set $(-)^{\vee}:=\Hom_R(-,\w_R)$.

We begin by recording the following result of Iyengar which simply asserts that the familiar depth lemma holds for complexes:
\begin{prop}[{see \cite[Proposition 5.1]{iyengar}}]\label{depthlemma}
Let $0 \to X \to Y \to Z \to 0$ be a short exact sequence of $R$-complexes. Then we have the following inequalities:
\begin{enumerate}
\item[$(1)$] $\depth_R(Y) \ge \min\{\depth_R(X),\depth_R(Z)\}$.
\item[$(2)$] $\depth_R(Z) \ge \min\{\depth_R(X)-1,\depth_R(Y)\}$.
\item[$(3)$] $\depth_R(X) \ge \min\{\depth_R(Y),\depth_R(Z)+1\}$.
\end{enumerate}

\end{prop}

We also note that we have the following familiar base independence property for depth of complexes:
\begin{prop}\label{baseindep}
Let $R$ be a local ring, let $S$ be a local $R$-algebra with maximal ideal $\m S$, and let $M$ be a complex of $S$-modules. Then $\depth_R(M)=\depth_{S}(M)$. In particular, we may take $S$ to be $\hat{R}$ or $R/I$ for any ideal $I$ in $R$.
\end{prop}

\begin{proof}
Let $y_1,\dots,y_n$ be a minimal generating set for $\m$, let $K_R$ be the Koszul complex on this generating set over $R$, and let $K_{S}$ be the Koszul complex on the images of the $y_i$ in $S$, so $K_{S} \cong K_R \otimes_R S$. Then $K_R \otimes_R M \cong K_R \otimes_R (S \otimes_{S} M) \cong (K_R \otimes_R S) \otimes_{S} M \cong K_{S} \otimes_{S} M$, and the claim follows from \cite[Theorem I]{FI03}.
\end{proof}

Going forward, we will use Proposition \ref{baseindep} when it applies tacitly and without reference.

The next Proposition extends to a key class of complexes the familiar fact for modules that $M$ is MCM if and only if $\depth_R(M) \ge \dim(R)$.

\begin{prop}\label{mcmcomplex}
Let $X$ be a complex of $R$-modules such that $\inf\{i \mid H_i(X) \ne 0\}=0$ and $H(X)$ is bounded. Then $X$ is MCM if and only if $\depth_R(X) \ge \dim(R)$. 
\end{prop}

\begin{proof}
The forward direction holds by definition. For the converse, it suffices to show that $H^i_{\m}(X)=0$ for $i>\dim(R)$. Let $I$ be an injective resolution of $X$. Since $H_i(X)=0$ when $i<0$, the complex 
\[0 \to \ker(\partial^I_0) \rightarrow I_0 \xrightarrow{\partial^I_0} I_{-1} \xrightarrow{\partial^I_{-1}} \cdots\]
is acyclic. Then for $i>0$, $H^i_{\m}(X)=H_{-i}(\Gamma_{\m}(I))=H^i_{\m}(\ker(\partial^I_0))$. In particular, $H^i_{\m}(X)=0$ for $i>\dim(R)$ by e.g. \cite[Theorem 3.5.6]{BH93}. 
\end{proof}

A key consequence of Proposition \ref{depthlemma} is the following:
\begin{prop}\label{replacesyz}
Suppose $M$ and $N$ are $R$-modules. Then $\codepth_R(M \otimes^L_R N) \le \codepth_R(M)+\codepth_R(N)$ if and only if $\Omega^{\codepth_R(M)}(M) \otimes^L_R \Omega^{\codepth_R(N)}(N)$ is MCM.
\end{prop}
\begin{proof}
Let $F^M$ and $F^N$ be a minimal free resolution of $M$ and $N$ respectively. Applying $- \otimes^L_R N$ to the short exact sequences \[0 \to \Omega^i_R(M) \to F^M_{i-1} \to \Omega^{i-1}_R(M) \to 0\]
for $1 \le i \le \codepth_R(M)$, we have short exact sequences
\[0 \to \Omega^i_R(M) \otimes^L_R N \to F^M_{i-1} \otimes^L_R N \to \Omega^{i-1}_R(M) \otimes^L_R N \to 0.\]
Similarly, we have short exact sequences
\[0 \to \Omega^{\codepth_R(M)}_R(M) \otimes^L_R \Omega^j_R(N) \to \Omega^{\codepth_R(M)}(M) \otimes^L_R F^N_{j-1} \to \Omega^{\codepth_R(M)}(M) \otimes^L_R \Omega^{j-1}_R(N) \to 0\]
for $1 \le j \le \codepth_R(N)$.
The claim now follows from Proposition \ref{depthlemma}.
\end{proof}

\begin{lem}\label{torextdual}
Suppose $M$ and $N$ are $R$-modules with $N$ and $M \otimes_R \Ext^{d-n}_R(N,\w_R)$ CM of dimension $n$. If $\Tor^R_{1 \le j \le i}(M,\Ext^{d-n}_R(N,\w_R))=0$ for some $i$, then $\Ext^{1 \le j \le i}_R(M,N)=0$.    
\end{lem}

\begin{proof}
If $L$ is a CM $R$-module, we set $L^{\dagger}:=\Ext^{\codepth_R(L)}_R(L,\w_R)$. Let $F:=F_{i+1} \rightarrow F_i \rightarrow \cdots F_0 \rightarrow 0$ be part of a free resolution of $M$. Applying $- \otimes_R N^{\dagger}$, we get, since $\Tor^R_{1 \le j \le i}(M,N^{\dagger})=0$, an exact sequence $F \otimes_R N^{\dagger} \rightarrow M \otimes_R N^{\dagger} \rightarrow 0$. Since each term of this sequence is CM of dimension $n$, the sequence $0 \rightarrow (M \otimes_R N^{\dagger})^{\dagger} \rightarrow (F \otimes_R N^{\dagger})^{\dagger}$ is exact. But this sequence is isomorphic to $0 \rightarrow \Hom_R(M,N) \rightarrow \Hom_R(F,N)$ whose homologies are $\Ext^j_R(M,N)$ for $1 \le j \le i$, so $\Ext^{1 \le j \le i}_R(M,N)=0$. 
\end{proof}

For our next result, we recall that the Auslander Transpose $\Tr_R(M)$ is defined as $\coker(A^T)$ where $A$ is a minimal presentation matrix for $M$. 
If $L$ is a CM $R$-module, we set $L^{\dagger}:=\Ext^{\codepth_R(L)}_R(L,\w_R)$. See \cite[Section 3.3]{BH93} as a reference for canonical duality in this level of generality. In particular, we note that if $\underline{s}$ is a maximal regular sequence in the annihilator of $L$, the $L^{\dagger}$ can be identified with $\Hom_{R/\underline{s}}(L,\w_{R/\underline{s}})$; see e.g. \cite[Lemma 3.1.16]{BH93}. The following directly extends \cite[Theorem 2.2]{KO22} by relaxing the MCM requirements.
\begin{prop}\label{transposeextend}
Suppose $M$ and $N$ are $R$-modules and that $N$ is CM of dimension $n$. Then $M \otimes_R N$ is CM of dimension $n$ if and only if $\Tr_R(M) \otimes_R \Ext^{d-n}_R(N,\w_R)$ is CM of dimension $n$. 
\end{prop}
\begin{proof}
Set $(-)^{\dagger}:=\Ext^{d-n}_R(-,\w_R)$. Let $\epsilon:R^{\oplus n} \xrightarrow{A} R^{\oplus m} \rightarrow M \rightarrow 0$ be a minimal presentation of $M$. Dualizing into $R$ and applying $- \otimes_R N^{\dagger}$ there is an exact sequence of the form $\delta:(N^{\dagger})^{\oplus m} \xrightarrow{A^T} (N^{\dagger})^{\oplus n} \rightarrow \Tr_R(M) \otimes_R N^{\dagger} \rightarrow 0$, while applying $\Hom_R(-,N^{\dagger})$ to $\epsilon$ instead give an exact sequence $0 \to \Hom_R(M,N^{\dagger}) \to (N^{\dagger})^{\oplus m} \xrightarrow{A^T} (N^{\dagger})^{\oplus n}$. Splicing these, we have an exact sequence
\[(*):0 \rightarrow \Hom_R(M,N^{\dagger}) \rightarrow (N^{\dagger})^{\oplus m} \xrightarrow{A^T} (N^{\dagger})^{\oplus n} \rightarrow \Tr(M) \otimes N^{\dagger} \rightarrow 0.\]
Similarly, tensoring  $\epsilon$ with $N$ gives an exact sequence $N^{\oplus n} \xrightarrow{A} N^{\oplus m} \rightarrow M \otimes_R N \rightarrow 0$ and applying $(-)^{\dagger}$ to $\delta$ gives an exact sequence $0 \rightarrow (\Tr_R(M) \otimes_R N^{\dagger})^{\dagger} \rightarrow N^{\oplus n} \xrightarrow{A} N^{\oplus m}$. Splicing yields and exact sequence
\[(**):0 \rightarrow (\Tr_R(M) \otimes_R N^{\dagger})^{\dagger} \rightarrow N^{\oplus n} \xrightarrow{A} N^{\oplus m} \rightarrow M \otimes_R N \rightarrow 0.\]

Suppose first that $M \otimes_R N$ is CM of dimension $n$. Then as all terms of $(**)$ are CM of dimension $n$ by the depth lemma, we may apply $(-)^{\dagger}$ to obtain an exact sequence
\[0 \rightarrow (M \otimes_R N)^{\dagger} \rightarrow (N^{\dagger})^{\oplus m} \xrightarrow{A^T} (N^{\dagger})^{\oplus n} \rightarrow (\Tr_R(M) \otimes_R N^{\dagger})^{\dagger \dagger} \rightarrow 0\] whose terms are all also CM of dimension $n$. Comparing with $(*)$, it follows that $\Tr_R(M) \otimes_R N^{\dagger} \cong (\Tr_R(M) \otimes_R N^{\dagger})^{\dagger \dagger}$, and is thus CM of dimension $n$.

So we have shown that if $M \otimes_R N$ is CM of dimension $n$, then $\Tr_R(M) \otimes_R N^{\dagger}$ is CM of dimension $n$. Then if $\Tr_R(M) \otimes_R N^{\dagger}$ is CM of dimension $n$, it follows that $\Tr_R(\Tr_R(M)) \otimes_R N^{\dagger \dagger} \cong \Tr_R(\Tr_R(M)) \otimes_R N$ is CM of dimension $n$. But $M \otimes_R N \cong (\Tr_R(\Tr_R(M)) \otimes_R N) \oplus (F \otimes_R N)$ for some free $R$-module $F$ by e.g. \cite[Proposition 2.2 (4)]{DT15}, so is also CM of dimension $n$, completing the proof.
\end{proof}

Using Proposition \ref{transposeextend}, we can give a fairly broad extension of \cite[Lemma 5.3]{DE21} (cf. \cite[Lemma 3.4]{LM20}) that will be needed in later sections.

\begin{prop}\label{mcmext}
Suppose $R$ is CM with canonical module $\w_R$ and let $M$ and $N$ be $R$-modules with $N$ CM of dimension $n$. If $\Ext^{1 \le i \le n}_R(M,N)=0$, then $M \otimes_R \Ext^{d-n}_R(N,\w_R)$ is CM of dimension $n$. The converse holds if either $\Ext^{i}_R(M,N)$ has finite length for $1 \le i \le n$ or $\Tor^R_i(M,\Ext^{d-n}_R(N,\w_R))$ has finite length for $1 \le i \le n$. 
\end{prop}

\begin{proof}
Suppose first that $\Ext^{1 \le i \le n}_R(M,N)=0$. Let $F:=F_{n+1} \xrightarrow{\partial_n} F_n \xrightarrow{\partial_{n-1}} \cdots \xrightarrow{\partial_0} F_0 \xrightarrow{p} M \rightarrow 0$ be part of a minimal free resolution of $M$. As $\Ext^{1 \le i \le n}_R(M,N)=0$, the complex
\[0 \rightarrow \Hom_R(M,N) \xrightarrow{\Hom(p,N)} \Hom_R(F_0,N) \xrightarrow{\Hom(\partial_0,N)} \Hom_R(F_1,N) \rightarrow \cdots \rightarrow \Hom_R(F_{d+1},N)\]
is acyclic. Split this exact sequence into two exact sequences
\[0 \rightarrow \Hom_R(M,N) \xrightarrow{\Hom(p,N)} \Hom_R(F_0,N) \xrightarrow{\Hom(\partial_0,N)} \Hom_R(F_1,N) \rightarrow P \rightarrow 0\]
and \[0 \rightarrow P \rightarrow \Hom_R(F_2,N) \rightarrow \cdots \rightarrow \Hom_R(F_{n+1},N),\]
and note that $P \cong \Tr_R(M) \otimes_R N$ as in the proof of Proposition \ref{transposeextend}.
It follows from the depth lemma that $P$ is CM of dimension $n$, and then so is $M \otimes_R N^{\dagger}$ by Proposition \ref{transposeextend}.

For the converse, there is nothing to prove if $n=0$, so we let $n>0$. Suppose first that $M \otimes_R N^{\dagger}$ is CM of dimension $n$ and that $\Ext^i_R(M,N)$ has finite length for $1 \le i \le n$. We proceed by induction on $n$. By prime avoidance, we may choose $y \in \bigcap_{1 \le i \le n} \Ann_R(\Ext^i_R(M,N))$ that is a nonzerodivisor on $R$, $N$, and $M \otimes_R N^{\vee}$. 

Then applying $\Hom_R(M,-)$ to this exact sequence we get exact sequences
\[0 \rightarrow \Hom_R(M,N) \xrightarrow{y} \Hom_R(M,N) \rightarrow \Hom_R(M,N/yN) \rightarrow \Ext^1_R(M,N) \rightarrow 0\]
and
\[0 \to \Ext^i_R(M,N) \to \Ext^i_R(M,N/yN) \to \Ext^{i+1}_R(M,N) \to 0\]
for $1 \le i \le n-1$.

In particular, $\Ext^i_R(M,N/yN)$ has finite length for $1 \le i \le n-1$. Since $N/yN$ and $M \otimes_R (N/yN)^{\dagger} \cong (M \otimes_R N^{\dagger})/y(M \otimes_R N^{\dagger})$ are CM of dimension $n-1$ it follows from inductive hypothesis that $\Ext^{1 \le i \le n-1}_R(M,N/yN)=0$. If $n>1$, then we have $\Ext^{1 \le i \le n}_R(M,N)=0$, giving the claim, so it remains to show the case where $n=1$. Applying $\Hom_R(-,\w_R)$ to the short exact sequence
\[0 \rightarrow N \xrightarrow{y} N \rightarrow N/yN \rightarrow 0\]
we get a short exact sequence
\[0 \rightarrow N^{\dagger} \xrightarrow{\cdot y} N^{\dagger} \rightarrow (N/yN)^{\dagger} \rightarrow 0,\]

and applying $M \otimes_R -$ to this exact sequence gives another exact sequence
\[0 \rightarrow M \otimes_R N^{\dagger} \xrightarrow{\cdot y} M \otimes_R N^{\dagger} \rightarrow M \otimes_R (N/yN)^{\dagger} \rightarrow 0.\]
Applying $\Hom_R(-,\w_R)$ to this exact sequence we get another exact sequence fitting into a commutative diagram
\[\begin{tikzcd}
	0 & {(M \otimes_R N^{\dagger})^{\dagger}} & {(M \otimes_R N^{\dagger})^{\dagger}} & {(M \otimes_R (N/yN)^{\dagger})^{\dagger}} & 0 \\
	0 & {\Hom_R(M,N)} & {\Hom_R(M,N)} & {\Hom_R(M,N/yN)} & {\Ext^1_R(M,N)} & 0
	\arrow[from=1-1, to=1-2]
	\arrow["{\cdot y}", from=1-2, to=1-3]
	\arrow[from=1-2, to=2-2]
	\arrow[from=1-3, to=1-4]
	\arrow[from=1-3, to=2-3]
	\arrow[from=1-4, to=1-5]
	\arrow[from=1-4, to=2-4]
	\arrow[from=2-1, to=2-2]
	\arrow["{\cdot y}", from=2-2, to=2-3]
	\arrow[from=2-3, to=2-4]
	\arrow[from=2-4, to=2-5]
	\arrow[from=2-5, to=2-6]
\end{tikzcd}\]
whose vertical arrows are isomorphisms induced by Hom-tensor adjointness, noting that for $R/(y)$-modules, $(-)^{\dagger}$ may be identified with $\Hom_{R/(y)}(-,\w_{R/(y)})$. In particular, $\Ext^1_R(M,N)=0$, so the claim holds when $n=1$.

Finally, if instead we suppose that $M \otimes_R N^{\dagger}$ is CM of dimension $n$ and that $\Tor^R_{1 \le i \le n}(M,N^{\dagger})$ has finite length, then localizing at any nonmaximal prime $\p$ and appealing to Lemma \ref{torextdual} shows that $\Ext^i_R(M,N)$ has finite length for $1 \le i \le n$, reducing to the case above and completing the proof.    
\end{proof}

The following is well-known to experts, but we provide the explicit statement and a short proof for the convenience of the reader:
\begin{prop}\label{exttorall}
Suppose $R$ is CM with canonical module $\w_R$. Let $M$ and $N$ be $R$-modules with $N$ MCM.
Then following are equivalent:
\begin{enumerate}
\item[$(1)$] $M \otimes_R N$ is MCM and $q_R(M,N)=0$.
\item[$(2)$] $b_R(M,N^{\vee})=0$.
\end{enumerate}
\end{prop}

\begin{proof}
We first show $(1) \Rightarrow (2)$. Let $F_M$ be a minimal free resolution of $M$. Since $q_R(M,N)=0$, $F_M \otimes_R N \to M \otimes_R N \to 0$ is acyclic. But as all the terms of this complex are MCM, we have that $0 \to (M \otimes_R N)^{\vee} \to (F_M \otimes_R N)^{\vee}$ is acyclic. But by Hom-tensor adjointness this complex is isomorphic to 
\[0 \to \Hom_R(M,N^{\vee}) \to \Hom_R(F_M,N^{\vee})\]
whose homologies are by definition $\Ext^i_R(M,N)$ in positive degree, so $b_R(M,N^{\vee})=0$.
For $(2) \Rightarrow (1)$, if $b_R(M,N^{\vee})=0$, then Proposition \ref{mcmext} forces $M \otimes_R N$ to be MCM, and we also have \[0 \to \Hom_R(M,N^{\vee}) \to \Hom_R(F_M,N^{\vee})\] is acylic. As in the previous part, Hom-tensor adjointness gives that $0 \to (M \otimes_R N)^{\vee} \to (F_M \otimes_R N)^{\vee}$ is acyclic. Applying $(-)^{\vee}$, then since each term of this complex is MCM, and since $N$ and $M \otimes_R N$ are as well, we get that $F_M \otimes_R N \to M \otimes_R N \to 0$ is acyclic. Thus $q_R(M,N)=0$, completing the proof. 
\end{proof}

\begin{lem}\label{cutdownMCM}

Suppose $M$ and $N$ are MCM $R$-modules. Then the following are equivalent:
\begin{enumerate}
\item[$(1)$] $M \otimes_R N$ MCM and $\Tor^R_{1 \le i \le d}(M,N)=0$.
\item[$(2)$] $\Tor^R_{1 \le i \le d}(M/\underline{x}M,N)=0$.
\item[$(3)$] $\tilde{M} \otimes_R N$ is MCM.
\end{enumerate}
If $R$ admits a canonical module $\w_R$, then these conditions are also equivalent to 
\begin{enumerate}
\item[$(4)$] $\Ext^{1 \le i \le d}_R(\tilde{M}),N^{\vee})=0$.
\end{enumerate}    
\end{lem}
\begin{proof}
We first show $(1) \Rightarrow (2)$. Suppose $M \otimes_R N$ MCM and $\Tor^R_{1 \le i \le d}(M,N)=0$. Let $F^M$ and $F^N$ be minimal free resolutions of $M$ and $N$ respectively. Then since $\Tor^R_{1 \le i \le d}(M,N)=0$, $F^M \otimes_R F^N$ has no homology in degrees $1 \le j \le d$. Thus truncating to these degrees gives part of a minimal free resolution of $M \otimes_R N$. Then in degrees $0 \le j \le d$, the homologies of $F^M \otimes_R F^N \otimes_R K(\underline{x},R)$ are $\Tor^R_j(M \otimes_R N,R/\underline{x}R)$, which are $0$ for $j\geq 1$ since $M \otimes_R N$ is MCM. But as $M$ is MCM, $F^M \otimes_R K(\underline{x},R)$ is a minimal free resolution of $M/\underline{x}M$ over $R$, and as $F^M \otimes_R F^N \otimes_R K(\underline{x},R) \cong (F^M \otimes_R K(\underline{x},R)) \otimes_R F^N$ has homologies $\Tor^R_j(M/\underline{x}M,N)$, the claim follows.

We now show $(2) \Rightarrow (3)$. Suppose $\Tor^R_{1 \le i \le d}(M/\underline{x}M,N)=0$. Let $F^{M/\underline{x}M}$ be the minimal free resolution of $M/\underline{x}M$ over $R$. Since $\Tor^R_{1 \le i \le d}(M/\underline{x}M,N)=0$, applying $- \otimes_R N$ to $F^{M/\underline{x}M}$ gives an exact sequence of the form
\[0 \to \tilde{M} \otimes_R N \to F_{d-1}^{M/\underline{x}M} \otimes_R N \to \cdots \to F^{M/\underline{x}M}_1 \otimes_R N \to F^{M/\underline{x}M}_0 \otimes_R N \to M/\underline{x}M \otimes_R N \to 0.\]
As $N$ is MCM, it follows from the depth lemma that $\tilde{M} \otimes_R N $ is MCM.

Next we show $(3) \Rightarrow (2)$. Suppose $\tilde{M} \otimes_R N$ is MCM. Again using $F^{M/\underline{x}M}$ for the minimal free resolution of $M/\underline{x}M$ over $R$, we have an exact sequence $0 \to \tilde{M} \to F_{d-1}^{M/\underline{x}M} \to \Omega^{d-1}_R(M/\underline{x}M) \to 0$ and applying $- \otimes_R N$ gives an exact sequence
\[0 \to \Tor^R_1(\Omega^{d-1}_R(M/\underline{x}M),N) \to \tilde{M} \otimes_R N \to F^{M/\underline{x}M}_{d-1} \otimes_R N \to \Omega^{d-1}_R(M/\underline{x}M) \otimes_R N \to 0.\]
Since $\Tor^R_1(\Omega^{d-1}_R(M/\underline{x}M),N) \cong \Tor^R_d(M/\underline{x}M,N)$ is killed by $\underline{x}$, and thus has finite length, and since $\tilde{M} \otimes_R N$ is MCM, it follows that $\Tor^R_1(\Omega^{d-1}_R(M/\underline{x}M),N)=0$ and the depth lemma then gives that $\depth_R(\Omega^{d-1}_R(M/\underline{x}M) \otimes_R N) \ge d-1$. Then we may repeat this argument successively on the lower syzygies to obtain that $\Tor^R_1(\Omega^{i-1}_R(M/\underline{x}M),N)=0$ for all $1 \le i \le d$, that is that $\Tor^R_i(M/\underline{x}M,N)=0$ for all $1 \le i \le d$, as desired.

Now we show $(2) \Rightarrow (1)$. Suppose $\Tor^R_{1 \le i \le d}(M/\underline{x}M,N)=0$. We will show inductively for all $0 \le j \le d$ that $\depth_R(M/(x_1,\dots,x_j)M \otimes_R M) \ge d-j$ and that $\Tor^R_{1 \le i \le d}(M/(x_1,\dots,x_j)M,N)=0$ for which the $j=0$ case concludes the proof. The base case where $j=d$ follows at once, so suppose for some $1 \le j \le d$, we have obtained that $\depth_R(M/(x_1,\dots,x_j)M \otimes_R N) \ge d-j$ and that $\Tor^R_{1 \le i \le d}(M/(x_1,\dots,x_j)M,N)=0$ for some $j$. Applying $- \otimes_R N$ to the short exact sequence 
\[0 \rightarrow M/(x_1,\dots,x_{j-1}) \xrightarrow{\cdot x_j} M/(x_1,\dots,x_{j-1})M \rightarrow M/(x_1,\dots,x_j)M \rightarrow 0\] we obtain from the induced the long exact sequence in $\Tor$ that the maps 
\[\Tor_i(M/(x_1,\dots,x_{j-1})M,N) \xrightarrow{\cdot x_j} \Tor^R_i(M/(x_1,\dots,x_{j-1})M,N)\] are surjective for all $1 \le i \le d$. By Nakayama's lemma, this forces $\Tor^R_{1 \le i \le d}(M/(x_1,\dots,x_{j-1})M,N)=0$. Further, as $\Tor^R_1(M/(x_1,\dots,x_{j})M,N)=0$, we get that $x_j$ is a nonzerodivisor on the module $M/(x_1,\dots,x_{j-1})M \otimes_R N$ so that $\depth_R(M/(x_1,\dots,x_{j-1})M \otimes_R N)=\depth_R(M/(x_1,\dots,x_j)M \otimes_R N)+1=d-j+1=d-(j-1)$. The claim then follows from induction, completing the proof.

Assuming $R$ admits a canonical module, then for $(3) \iff (4)$, since $\tilde{M}$ is locally free on the punctured spectrum, $\Tor^R_i(\tilde{M},N)$ has finite length for all $i>0$, and the proof is completed by Proposition \ref{mcmext}.
\end{proof}

\begin{lem}\label{torcutdown}

Suppose $N$ is an $R$-module and $M$ is an MCM $R$-module. Then
\begin{enumerate}
\item[$(1)$] $q_R(\tilde{M},N) \le q_R(M,N)$, and
\item[$(2)$] $b_R(\tilde{M},N)=b_R(M,N)$.
\end{enumerate}

\end{lem}

\begin{proof}
Set $q:=q_R(M,N)$. We proceed by induction on $j$ to show $\Tor^R_{i>q+j}(M/(x_1,\dots,x_j)M,N)=0$ for all $0 \le j \le d$ with the base case when $j=0$ holding by assumption. Suppose we have the claim for some $j$. Applying $- \otimes_R N$ to the short exact sequence
\[0 \rightarrow M/(x_1,\dots,x_j)M \xrightarrow{\cdot x_{j+1}} M/(x_1,\dots,x_j)M \rightarrow M/(x_1,\dots,x_{j+1})M \rightarrow 0\]
we observe from the induced long exact sequence in $\Tor$ that $\Tor^R_{i>q+j+1}(M/(x_1,\dots,x_{j+1}M,N)=0$. That $\Tor^R_{i>q+d}(M/\underline{x}M,N)=0$ follows from induction, and claim $(1)$ follows from dimension shifting. 

For claim $(2)$, it suffices to show for any $b  \ge 0$ that $\Ext^{i>b}_R(M,N)=0$ if and only if we have $\Ext^{i>b+d}_R(M/\underline{x}M,N)=0$. 
To see this, we show the stronger claim that $\Ext^{i>b}_R(L,N)=0$ if and only if $\Ext^{i>b+1}_R(L/xL,N)=0$ for any $R$-module $L$ and any $x$ that is regular on $L$.

Apply $\Hom_R(-,N)$ to the short exact sequence
\[0 \rightarrow L \xrightarrow{\cdot x} L \rightarrow L/xL \rightarrow 0\]
and consider the induced long exact sequence in $\Ext$. If $\Ext^{i>b}_R(L,N)=0$, then we immediately get that $\Ext^{i>b+1}_R(L/xL,N)$. Conversely, if $\Ext^{i>b+1}_R(L/xL,N)=0$, then the maps $\Ext^i_R(L,N) \xrightarrow{\cdot x} \Ext^i_R(L,N)$ are surjective for $i>b$, and then $\Ext^{i>b}_R(L,N)=0$ by Nakayama's lemma, completing the proof.
\end{proof}

\begin{lem}\label{negativeqr}
Suppose $M$ is an $R$-complex with $t:=\sup\{i \mid H_i(M) \ne 0\}<\infty$. Then $\depth_R(M) \ge -t$ with equality if and only if $\depth_R(H_t(M))=0$. 
\end{lem}

\begin{proof}
Since $t<\infty$, it follows that $M$ has an injective resolution $I$ for which $\sup(I)=t$ (see \cite[1.2 (I)]{CJ15}). If $i>t$, then $H_i(\Hom_R(k,I))=0$, so $\depth_R(M) \ge -t$. But we have equality if and only if $H_t(\Hom_R(k,I)) \cong \Hom_R(k,H_t(I)) \cong \Hom_R(k,H_t(M)) \ne 0$, that is, if $\depth_R(H_t(M))=0$, establishing the claim.
\end{proof}

\subsection{Comparing the Different Versions of $\ldep$ and $\rdep$}\label{subsection}
In this subsection, we describe the relationship between the derived $\ldep$/$\rdep$ conditions and their counterparts for modules. We present them here rather than in their devoted sections since their arguments are nearly identical.

\begin{lem}\label{depthcutdown}
If $x \in \m$, and if $N$ is an $R$-complex, then $\depth_R(N \otimes_R K(x,R))=\depth_R(N)-1$ where $K(x,R)$ denotes the Koszul complex on $x$. In particular, if $x$ is regular on $R$, then $\depth_R(N \otimes^L_R R/x)=\depth_R(N)-1$.
\end{lem}

\begin{proof}
Applying $\RHom_R(k,-)$ to the exact triangle
\[N \xrightarrow{x} N \rightarrow N \otimes_R K(x,R) \rightarrow N[1]\]
gives another exact triangle
\[\RHom_R(k,N) \xrightarrow{x} \RHom_R(k,N) \rightarrow \RHom_R(k,N \otimes_R K(x,R)) \rightarrow \RHom_R(k,N)[1].\]

Since $x \in \m$, the induced long exact sequence in homology breaks into short exact sequences
\[0 \rightarrow H^i(\RHom_R(k,N)) \rightarrow H^i(\RHom_R(k,N \otimes_R K(x,R))) \rightarrow H^{i+1}(\RHom_R(k,N)) \rightarrow 0,\]
and the claim follows from the definition of $\depth$.

\end{proof}

We make use of the following lemma:
\begin{lem}\label{depthcopies}
Let $F:0 \to F_a \to F_{a-1} \to \cdots \to F_b \to 0$ be a complex of free $R$-modules and let $M$ be an $R$-complex. Then $\depth_R(M \otimes^L_R F) \ge \depth_R(M)-a$. 
\end{lem}

\begin{proof}
We proceed by induction on $a-b$. The claim is clear if $a-b=0$. If $a-b>0$, then we have an exact sequence of complexes
\[0 \to \tilde{F} \to F \rightarrow \tau_{\ge a}(F) \to 0\]
where $\tilde{F}:=0 \to F_{a-1} \to \cdots \to F_b \to 0$, which in turn induces an exact sequence
\[0 \to M \otimes^L_R \tilde{F} \to M \otimes^L_R F \to M \otimes^L_R \tau_{\ge a}(F) \to 0.\]. From induction hypothesis we have $\depth_R(M \otimes^L_R \tilde{F}) \ge \depth_R(M)-a+1$ and from the base case we have $\depth_R(M \otimes^L_R \tau_{\ge a}(F))=\depth_R(F_a) \ge \depth_R(M)-a$. It follows from Proposition \ref{depthlemma} that $\depth_R(M \otimes^L_R F) \ge \depth_R(M)-a$, and the claim follows.
\end{proof}

\begin{theorem}\label{derivedldepvsformodules}
Suppose $R$ is a local ring with $t:=\depth(R)$. Then
\begin{enumerate}
\item[$(1)$] $R$ satisfies derived $\ldep$ if and only if $R$ satisfies derived $\ldep$ for modules.
\item[$(2)$] If $t>0$, then $R$ satisfies derived $\rdep$ if and only if $R$ satisfies derived $\rdep$ for modules.
\end{enumerate}
\end{theorem}

\begin{proof}
Derived $\ldep$ (resp. derived $\rdep$) clearly implies derived $\ldep$ for modules (resp. derived $\rdep$ for modules) and  as noted in Remark \ref{derivedimplies}. For the converses, we begin with a general setup applicable to both $(1)$ and $(2)$. Let $M$, $N$ be $R$-complexes for which $M$, $N$ and $M \otimes^L_R N$ have bounded homology. By shifting, we may suppose $\sup\{i \mid H_i(M) \ne 0\}=\sup\{i \mid H_i(N) \ne 0\}=0$. Let $F$ and $G$ be right bounded free resolutions of $M$ and $N$ respectively.

Consider the natural short exact sequence
\[\textbf{(i)}:0 \to \tilde{G} \to G \to \tau_{\ge 0}(G) \to 0.\]
Then $\tau_{\ge 0}(G) \simeq C_N$ where $C_N$ is an $R$-module, and we have a short exact sequence
\[\textbf{(ii)}: 0 \to M \otimes^L_R \tilde{G} \to M \otimes^L_R N \to M \otimes^L_R C_N \to 0.\]
It follows from the long exact sequence in homology induced by $\textbf{(ii)}$ that $M \otimes^L_R C_N$ has bounded homology. From Lemma \ref{depthcopies}, we have $\depth_R(\tilde{G}) \ge t+1$ and $\depth_R(M \otimes^L_R \tilde{G}) \ge \depth_R(M)+1$. If $\depth_R(N)>0$, then it follows from Lemma \ref{negativeqr} that $\depth_R(H_0(N))>0$. Let $x \in \m$ be a nonzerodivisor on $H_0(N)$ and let $K(x,R)$ denote the Koszul complex on $x$. The exact triangle 
\[N \xrightarrow{x} N \rightarrow N \otimes_R K(x,R) \rightarrow N[1]\]
induces the long exact sequence
\[\begin{tikzpicture}[descr/.style={fill=white,inner sep=1.5pt}]
        \matrix (m) [
            matrix of math nodes,
            row sep=1em,
            column sep=1.8em,
            text height=1.5ex, text depth=0.25ex
        ]
        { H_0(N) & H_0(N) & H_0(N \otimes_R K(x,R)) & \cdots \\
           H_1(N) & H_1(N) & H_1(N \otimes_R K(x,R)) & \mbox{} \\
          H_2(N)  & H_2(N) & H_2(N \otimes_R K(x,R)) & \mbox{} \\
          \mbox{}  & \mbox{} & \hspace{2cm} & \mbox{} \\
        };

        \path[overlay,->, font=\scriptsize,>=latex]
        (m-1-1) edge node[above]{$\cdot x$} (m-1-2)
        (m-1-2) edge (m-1-3)
        (m-1-3) edge (m-1-4);
        \path[overlay,->, font=\scriptsize,>=latex]
        (m-2-3) edge[out=365,in=185] (m-1-1)
        (m-2-2) edge (m-2-3)
        (m-2-1) edge (m-2-2)
        (m-3-3) edge[out=365,in=185] (m-2-1);
        \path[overlay,->, font=\scriptsize,>=latex]
        (m-3-2) edge (m-3-3)
        (m-3-1) edge (m-3-2);
        \path[overlay,->, font=\scriptsize,>=latex]
        (m-4-3) edge[out=365,in=185,dashed] (m-3-1);   
\end{tikzpicture}\]
Since $H_i(N)=0$ for $i>0$ and since $H_0(N) \xrightarrow{ \cdot x} H_0(N)$ is injective, it follows from Nakayama's lemma that $\sup\{i \mid H_i(N \otimes_R K(x,R)) \ne 0\}=0$. By Lemma \ref{depthcutdown}, we have $\depth_R(N \otimes_R K(x,R))=\depth_R(N)-1$ and $\depth_R(M \otimes^L_R N \otimes_R K(x,R))=\depth_R(M \otimes^L_R N)-1$.

To prove the desired claims, we may apply the argument above iteratively to suppose $\depth_R(N)=0$. Similarly, consider the exact sequence \[0 \to \tilde{F} \to F \to \tau_{\ge 0}(F) \to 0,\]
we have $\tau_{\ge 0}(F) \simeq C_M$ where $C_M$ is an $R$-module, and we may apply a similar argument to suppose that $\depth_R(M)=0$. Then it follows from Proposition \ref{depthlemma} that $\depth_R(C_M)=\depth_R(C_N)=0$.

Next, we have exact sequences
\[\textbf{(iii)}: 0 \to \tilde{F} \otimes^L_R N \to M \otimes^L_R N \to C_M \otimes^L_R N \to 0\] and
\[\textbf{(iv)}: 0 \to C_M \otimes^L_R \tilde{G} \to C_M \otimes^L_R N \to C_M \otimes^L_R C_N \to 0.\]
It follows from the long exact sequences in homology induced by these that $C_M \otimes^L_R N$ and $C_M \otimes^L_R C_N$ have bounded homology, so that $q_R(C_M,C_N)<\infty$, and we have $\depth_R(\tilde{F}) \ge t+1$, $\depth_R(\tilde{F} \otimes^L_R N) \ge 1$, and $\depth_R(C_M \otimes^L_R \tilde{G}) \ge 1$ from Lemma \ref{depthcopies}.

Now, for $(1)$, if $R$ satisfies derived $\ldep$ for modules then $\depth_R(C_M \otimes^L_R C_N) \ge -t$. Applying Proposition \ref{depthlemma} to $\textbf{(iv)}$ gives that $\depth_R(C_M \otimes^L_R N) \ge -t$, and applying it to $\textbf{(iii)}$ then gives that $\depth_R(M \otimes^L_R N) \ge -t=0+0-t$. Thus $R$ satisfies derived $\ldep$.

For $(2)$, if $R$ satisfies derived $\rdep$ for modules and $t>0$, then $\depth_R(C_M \otimes^L_R C_N) \le -t<0$. Then applying Proposition \ref{depthlemma} to $\textbf{(iv)}$, we see that $\depth_R(C_M \otimes^L_R N)=\depth_R(C_M \otimes^L_R C_N) \le -t$. Similarly, applying Lemma \ref{depthlemma} to $\textbf{(iii)}$ gives that $\depth_R(M \otimes^L_R N)=\depth_R(C_M \otimes^L_R N) \le -t=0+0-t$. Thus $R$ satisfies derived $\rdep$, completing the proof.
\end{proof}

\section{The $\ldep$ and Derived $\ldep$ conditions}\label{ldepsection}

Our primary goal in this section is to prove Theorem \ref{ldepintro}. Before proceeding, we need to understand how the derived $\ldep$ condition behaves with respect to modding out regular sequences.

\begin{prop}\label{ldepcutsdown}
Suppose $R$ is a local ring and that $x \in \m$ is a nonzerodivisor. Then $R$ satisfies derived $\ldep$ if and only if $R/x$ satisfies derived $\ldep$.
\end{prop}

We note in particular that Proposition \ref{ldepcutsdown} does not require $R$ to be CM. Before proceeding with the proof of Proposition \ref{ldepcutsdown}, we need the following Lemma which allows for a change of rings argument in a higher level of generality than can be easily obtained from the standard spectral sequence argument. This lemma will also be used for a similar analysis on the derived $\rdep$ condition in the next section:
\begin{lem}\label{ses}
Let $R$ be a local ring and let $x \in \m$ be a nonzerodivisor. Suppose $M$ and $N$ are $R/x$-complexes with bounded below homology. Then there is an exact sequence
\[0 \to (M \otimes^L_{R/x} N)[-1] \to M \otimes^L_R N \to M \otimes^L_{R/x} N \to 0.\]
\end{lem}

\begin{proof}
Let $F^{R/x}_N$ be a bounded below complex of finitely generated free $R/x$-modules that is quasi-isomorphic to $N$. Let $(F,\partial^F)$ be a lift of $F^{R/x}_N$ to $R$, so $F$ is a graded $R$-module whose components are finitely generated free modules over $R$, and $d$ is a degree $-1$ graded $R$-endomorphism of $F$. It follows from \cite[Section 3]{BJ20} that there is a degree $-2$ endomorphism $t^e$ of $F$ so that we have a complex
\[F^R_N:= \cdots \rightarrow F_n \oplus F_{n+1} \xrightarrow{\begin{pmatrix} \partial_n^F & (-1)^{n+1}t^e \\ (-1)^nx & \partial^F_{n+1} \end{pmatrix}} F_{n-1} \oplus F_n \xrightarrow{\begin{pmatrix} \partial_{n-1}^F & (-1)^{n}t^e \\ (-1)^{n-1}x & \partial^F_{n} \end{pmatrix}} F_{n-2} \oplus F_{n-1} \rightarrow \cdots\]
that is quasi-isomorphic to $N$, and moreover that $\bar{t^e}:F^{R/x}_N \to F^{R/x}_N$ is a chain map. It follows that $R/x \otimes_R F^R_N \cong \cone(\bar{t^e})$ and there is a thus a short exact sequence of $R/x$-complexes
\[0 \to F^{R/x}_N[-1] \to R/x \otimes_R F^R_N \to F^{R/x}_N \to 0.\]
The claim follows from applying $M \otimes^L_{R/x} -$ and noting that $M \otimes^L_{R/x} (R/x \otimes_R F_N) \simeq M \otimes^L_{R} N$.  
\end{proof}

\begin{proof}[Proof of Proposition \ref{ldepcutsdown}]
First suppose $R/x$ satisfies derived $\ldep$ and let $M$ and $N$ be $R$-complexes such that $M,N$, and $M \otimes^L_R N$ have bounded homology. From the exact sequences
\[0 \rightarrow M \xrightarrow{\cdot x} M \rightarrow M \otimes^L_R R/x \rightarrow 0\]
and 
\[0 \rightarrow N \xrightarrow{\cdot x} N \rightarrow N \otimes^L_R R/x \rightarrow 0\]
we see that $M \otimes^L_R R/x$ and $N \otimes^L_R R/x$ have bounded homology. Moreover, we have $(M \otimes^L_R R/x) \otimes^L_{R/x} (N \otimes^L_R R/x) \simeq (M \otimes^L_R N) \otimes^L_R R/x$ which also has bounded homology from the short exact sequence
\[0 \rightarrow M \otimes^L_R N \xrightarrow{\cdot x} M \otimes^L_R N \rightarrow (M \otimes^L_R N) \otimes^L_R R/x \rightarrow 0.\]
Since $R/x$ satisfies derived $\ldep$, we have
\[\depth_{R/x}((M \otimes^L_R N) \otimes^L_R R/x)=\depth_{R/x}((M \otimes^L_R R/x) \otimes^L_{R/x} (N \otimes^L_R R/x))\]
\[\ge \depth_{R/x}(M \otimes^L_R R/x)+\depth_{R/x}(N \otimes^L_R R/x)-\depth(R/x)\]
\[=\depth_{R/x}(M \otimes^L_R R/x)+\depth_{R/x}(N \otimes^L_R R/x)-\depth(R)+1.\]
But then it follows from Lemma \ref{depthcutdown} that
\[\depth_R(M \otimes^L_R N) \ge \depth_R(M)+\depth_R(N)-\depth(R),\]
so $R$ satisfies derived $\ldep$.

 For the converse, suppose $R$ satisfies derived $\ldep$, and let $M$ and $N$ be $R/x$-complexes for which $M$, $N$, and $M \otimes^L_{R/x} N$ have bounded homology. From Lemma \ref{ses} we have a short exact sequence
\[(*):0 \to (M \otimes^L_{R/x} N)[-1] \to M \otimes^L_R N \to M \otimes^L_{R/x} N \to 0.\]
It follows the exact sequence (*) that $M \otimes^L_R N$ has bounded homology, and since $R$ satisfies derived $\ldep$, we have 
\[\depth_R(M \otimes^L_R N) \ge \depth_R(M)+\depth_R(N)-\depth(R).\]
But applying Proposition \ref{depthlemma} to (*), we see that $\depth_R(M \otimes^L_R N)=\depth_{R/x}(M \otimes^L_{R/x} N)-1$, so
\[\depth_R(M \otimes^L_{R/x} N) \ge \depth_{R}(M)+\depth_{R}(N)-\depth(R)+1=\depth_{R/x}(M)+\depth_{R/x}(N)-\depth(R/x),\]
and it follows that $R/x$ satisfies derived $\ldep$.
\end{proof}

Combining Proposition \ref{ldepcutsdown} with \cite[Lemma 2.10]{KL23} we have as an immediate consequence that derived $\ldep$ behaves well under completion.

\begin{cor}\label{ldepcompletes}
If $R$ is a local ring, then $R$ satisfies derived $\ldep$ if and only if $\hat{R}$ satisfies derived $\ldep$. 
\end{cor}

We are now ready to prove the main theorem of this section:

\begin{theorem}\label{ldepthm}
Let $R$ be a CM ring of dimension $d$. Consider the following conditions:
\begin{enumerate}
\item[$(1)$] $R$ satisfies derived $\ldep$. 
\item[$(2)$] $R$ satisfies derived $\ldep$ for modules. 
\item[$(3)$] For all finite length modules $A,B$, and setting $M:=\Omega^d_R(A)$ and $N:=\Omega^d_R(B)$, if $q_R(M,N)<\infty$, then $M \otimes^L_R N$ is MCM.
\item[$(4)$] If $M,N$ are MCM $R$-modules with $q_R(M,N)<\infty$, then $q_R(M,N)=0$ and $M \otimes_R N$ is MCM.
\item[$(5)$] If $M$ and $N$ are $R$-modules with $b_R(M,N)<\infty$, then $b_R(M,N) \le \codepth_R(M)$. 
\item[$(6)$] $R$ satisfies $\uac$ with $b_R=d$.
\item[$(7)$] For MCM $R$-modules $M$ and $N$, if $b_R(M,N)<\infty$, then $b_R(M,N)=0$.
\item[$(8)$] $R$ satisfies $\ldep$.
    
\end{enumerate}
Then we have the following:
\begin{enumerate}[label=(\upshape{\Roman*})]
\item Conditions $(1)-(6)$ are equivalent.
\item Conditions $(1)-(6)$ imply $(7)$ and the converse holds if $R$ admits a canonical module $\w_R$. 
\item Conditions $(1)-(6)$ imply $(8)$ and the converse holds if $\dim(R)>0$. 
    
\end{enumerate}
\end{theorem}

\begin{proof}

We first note that $(1) \iff (2)$ is the content of Theorem \ref{derivedldepvsformodules}, while $(2) \Rightarrow (3)$ is clear from Proposition \ref{mcmcomplex}, so we turn our attention to $(3) \Rightarrow (4)$. Assuming the condition of $(3)$, suppose $M$ and $N$ are MCM $R$-modules with $q_R(M,N)<\infty$. Note the claim is clear from Lemma \ref{negativeqr} if $d=0$, so we may suppose $d>0$.  As in the previous section, we let $\underline{x}$ be a maximal regular sequence in $R$ and set $\tilde{M}:=\Omega^d_R(M/\underline{x}M)$ and $\tilde{N}:=\Omega^d_R(N/\underline{x}N)$. By Lemma \ref{torcutdown}, we have that $q_R(\tilde{M},\tilde{N})<\infty$. Then the condition of $(3)$ forces $\tilde{M} \otimes^L_R \tilde{N}$ to be MCM. If $q_R(\tilde{M},\tilde{N})>0$, then since $\tilde{M}$ is locally free on the punctured spectrum, $\depth_R(\Tor^R_{q_R(\tilde{M},\tilde{N})}(\tilde{M},\tilde{N}))=0$, and Lemma \ref{negativeqr} would force $\depth_R(\tilde{M} \otimes^L_R \tilde{N})=-q_R(\tilde{M},\tilde{N})<0$, contradicting that $\tilde{M} \otimes^L_R \tilde{N}$ is MCM. It follows that $q_R(\tilde{M},\tilde{N})=0$, which gives as well that $\tilde{M} \otimes_R \tilde{N}$ is MCM. Then Lemma \ref{cutdownMCM} gives that $\Tor^R_{1 \le i \le d}(M,N)=0$ and that $M \otimes_R N$ is MCM. But we may repeat the argument above replacing $N$ by $\Omega^j_R(N)$ for any $j \ge 0$, which shows that $\Tor^R_{1 \le i \le d}(M,\Omega^j_R(N)) \cong \Tor^R_{1+j \le i \le d+j}(M,N)=0$ for all $j$. Thus $q_R(M,N)=0$ as desired.

We now show $(4) \Rightarrow (2)$. Suppose the condition of $(4)$ and suppose $M$ and $N$ are $R$-modules with $q_R(M,N)<\infty$. Then $q_R(\Omega^{\codepth_R(M)}_R(M),\Omega^{\codepth_R(N)}_R(N))<\infty$, so the condition of $(4)$ forces $q_R(\Omega^{\codepth_R(M)}_R(M),\Omega^{\codepth_R(N)}_R(N))=0$ and that $\Omega^{\codepth_R(M)}_R(M) \otimes_R \Omega^{\codepth_R(N)}_R(N)$ is MCM. In particular, this means $\Omega^{\codepth_R(M)}_R(M) \otimes^L_R \Omega^{\codepth_R(N)}_R(N)$ is MCM and the claim follows from Proposition \ref{replacesyz}.
We have thus established the equivalence of $(1)-(4)$.

We now show $(4) \Rightarrow (7)$. We note that the condition of $(3)$ obviously ascends to and descends from the completion since $\Omega^d_R(A) \otimes_R \hat{R} \cong \Omega^d_{\hat{R}}(A)$ for any finite length module $A$, while the condition of $(6)$ obviously descends from the completion. So we may suppose $R$ is complete and so in particular that $R$ admits a canonical module $\w_R$. Suppose $b_R(M,N)<\infty$. Then $b_R(\Omega^{b_R(M,N)}_R(M),N)=0$ and Proposition \ref{exttorall} implies that $q_R(\Omega^{b_R(M,N)}_R(M),N^{\vee})=0$. It follows that $q_R(M,N^{\vee})<\infty$ and then the condition of $(3)$ forces $q_R(M,N^{\vee})=0$ and that $M \otimes_R N$ is MCM. But then appealing to Proposition \ref{exttorall} again gives that $b_R(M,N)=0$. 

Next we show $(7) \Rightarrow (4)$ under the assumption that $R$ admits a canonical module. Suppose $q_R(M,N)<\infty$. Then $q_R(\Omega^{q_R(M,N)}_R(M),N)=0$. Applying $- \otimes_R N$ to part of a minimal free resolution $F$ of $\Omega^{q_R(M,N)}_R(M)$, we get an exact sequence of the form
\[0 \to \Omega^{q_R(M,N)+d}_R(M) \otimes_R N \to F_d \otimes_R N \to \cdots \to F_0 \otimes_R N \to \Omega^{q_R(M,N)}_R(M) \otimes_R N \to 0.\]
It follows from the depth lemma that $\Omega^{q_R(M,N)+d}_R(M) \otimes_R N$ is MCM. Then Proposition \ref{exttorall} implies that $b_R(\Omega^{q_R(M,N)+d}_R(M),N)=0$, and it follows that $b_R(M,N^{\vee})<\infty$. Then the condition of $(5)$ forces $b_R(M,N^{\vee})=0$, and Proposition \ref{exttorall} forces $q_R(M,N)=0$ and that $M \otimes_R N$ is MCM.

Next we show that $(5)$ is equivalent $(1)-(4)$. To see this, since $(3)$ ascends to and descends from the completion, so too do conditions $(1)$, $(2)$, and $(4)$, while it is known from \cite[Remark 5.7]{CH10} that $(5)$ ascends to and descends from the completion as well. We may thus suppose $R$ is complete so that it admits a canonical module. We first show $(1) \Rightarrow (5)$. Suppose $M$ and $N$ are $R$-modules with $b_R(M,N)<\infty$. We have 
\[\RHom_R(M,N) \simeq \RHom_R(M,\RHom_R(\RHom_R(N,\w_R),\w_R)) \simeq  \RHom_R(M \otimes^L_R \RHom_R(N,\w_R),\w_R).\]
Since $N$ and $\RHom_R(M,N)$ have bounded homology, so do $\RHom_R(N,\w_R)$ and $M \otimes^L_R \RHom_R(N,\w_R)$.  Applying local duality \cite[3.4.1]{IM21}, we see that $\depth_R(\RHom_R(N,\w_R))=d$, and from the derived $\ldep$ condition of $(1)$, we have $\depth_R(M \otimes^L_R \RHom_R(N,\w_R)) \ge \depth_R(M)+\depth_R(\RHom_R(N,\w_R))-d=\depth_R(M)$. Applying local duality again, we get $b_R(M,N) \le \codepth_R(M)$.

We now show $(5) \Rightarrow (3)$. Suppose $A$ and $B$ are finite length $R$-modules, set $M:=\Omega^d_R(A)$ and $N:=\Omega^d_R(B)$, and suppose $q_R(M,N)<\infty$. Then as $N$ is MCM, we have $\RHom_R(M \otimes^L_R N,\w_R) \simeq \RHom_R(M,\RHom_R(N,\w_R)) \simeq \RHom_R(M,N^{\vee})$, which has bounded homology since $M \otimes^L_R N$ does, i.e., $b_R(M,N^{\vee})<\infty$. But then $b_R(M,N^{\vee})=b_R(A,N^{\vee})+d$, and as $b_R=d$, we must have $b_R(M,N^{\vee})=0$. It follows from Proposition \ref{exttorall}, that $q_R(M,N)=0$ and $M \otimes_R N$ is MCM, so $M \otimes^L_R N$ is MCM as desired. 

Note that $(5) \Rightarrow (6)$ is clear, since $\codepth_R(M) \le d$ for any $R$-module $M$. Conversely, suppose $b_R=d$, and take $M$,$N$ to be $R$-modules with $b_R(M,N)<\infty$. If $t:=b_R(M,N)>\codepth_R(M)$, then letting $L=\Omega^{\codepth_R(M)}_R(M)$, we have $0<b_R(L,N)<\infty$ and that $L$ is MCM. Then from Lemma \ref{torcutdown}, we have $b_R(L,N)=b_R(\tilde{L},N)=b_R(L/\underline{x}L,N)+d>d$, contradicting that $b_R=d$, so we now have that $(1)-(6)$ are equivalent. We also note that $(5)$ clearly implies $(7)$, so we have established items $(I)$ and $(II)$. 

We now show $(7) \Rightarrow (5)$ under the assumption that $R$ admits a canonical module $\w_R$. Suppose $M$ and $N$ are $R$-modules with $\codepth_R(M)<b_R(M,N)<\infty$. Then as $b_R(\Omega^{\codepth_R(M)}_R(M),N)=b_R(M,N)-\codepth_R(M)$, it suffices to show the claim when $M$ is MCM. Since $R$ has a canonical module, we may take an MCM approximation of $N$ (see \cite[Proposition 11.3]{LW12}), that is, a short exact sequence
\[0 \to Y \to L \to N \to 0\]
with $\id_R(Y)<\infty$ and $L$ MCM. 
Since $M$ is MCM and since $\id_R(Y)<\infty$, it follows (see e.g \cite[Definition 11.8]{LW12}) that $b_R(M,Y)=0$. Then applying $\Hom_R(M,-)$ to this exact sequence and considering the long exact sequence in $\Ext$ shows that $\Ext^i_R(M,L) \cong \Ext^i_R(M,N)$ for all $i>0$. In particular, $0<b_R(M,L)<\infty$, contradicting the condition of $(7)$.

Note that $(1) \Rightarrow (7)$ follows from Remark \ref{derivedimplies}.

Then to conclude the proof, we show $(7) \Rightarrow (4)$ when $d>0$. Indeed, if $M$ and $N$ are MCM $R$-modules with $q_R(M,N)<\infty$, then it suffices to show that $q_R(M,N)=0$, and if this is not the case then we may replace $M$ by $\Omega^{q_R(M,N)-1}_R(M)$ to suppose that $q_R(M,N)=1$. It follows from Lemma \ref{torcutdown} that $q_R(\tilde{M},N) \le 1$. In particular, $q_R(\Omega^1_R(\tilde{M}),N)=0$, and as $R$ satisfies $\ldep$, it follows that $\Omega^1_R(\tilde{M}) \otimes_R N$ is MCM. Applying $- \otimes_R N$ to the short exact sequence 
\[0 \to \Omega^1_R(\tilde{M}) \to R^{\oplus \mu_R(\tilde{M})} \to \tilde{M} \to 0\]
we observe there is an inclusion $\Tor^R_1(\tilde{M},N) \hookrightarrow \Omega_R^1(\tilde{M}) \otimes_R N$. But $\Tor^R_1(\tilde{M},N)$ has finite length since $\tilde{M}$ is locally free on the punctured spectrum of $R$, while $\Omega^1_R(M) \otimes_R N$ is MCM, and thus has positive depth since $d>0$. It can only be that $\Tor^R_1(\tilde{M},N)=0$ so that $q_R(\tilde{M},N)=0$. Since $R$ satisfies $\ldep$, it follows that $\tilde{M} \otimes_R N$ is MCM, and the claim follows from Proposition \ref{cutdownMCM}. 
\end{proof}

\begin{cor}\label{ldepcomplete}
Suppose $R$ is CM with $\dim(R):=d$ and let $x \in R$ be a regular element. Then the following hold:
\begin{enumerate}
\item[$(1)$] $R$ satisfies derived $\ldep$ for modules if and only if $R/(x)$ satisfies derived $\ldep$ for module. 
\item[$(2)$] If $R$ satisfies $\ldep$, then $R/(x)$ satisfies $\ldep$. The converse holds if $d>1$.
\item[$(3)$] $R$ satisfies derived $\ldep$ for modules if and only if $\hat{R}$ does.
\item[$(4)$] $R$ satisfies $\ldep$ if and only if $\hat{R}$ does.
\end{enumerate}

\end{cor}

\begin{proof}

The claims follow immediately from combining Theorem \ref{ldepthm} with Corollary \ref{ldepcompletes}.
\end{proof}

\begin{cor}\label{ldepimpliestr}
Suppose $R$ CM with $\dim(R)>0$. If $R$ satisfies $\ldep$, then $R$ satisfies $\tr$.
\end{cor}

\begin{proof}
If $R$ satisfies $\ldep$, then it follows from Theorem \ref{ldepthm} that $R$ satisfies $\uac$. That $R$ satisfies $\tr$ follows from \cite[Theorem C]{CH10}. 
\end{proof}

\begin{cor}\label{gorcase}
Suppose $R$ is Gorenstein. Then $\ldep$ implies $\rdep$. 
\end{cor}
\begin{proof}
There is nothing to prove if $\dim(R)=0$, so we may suppose $\dim(R)>0$. Then by Theorem \ref{ldepthm}, $R$ is an AB-ring. But $\dep$ is known to hold for AB-rings from \cite[Corollary 5.3 (b)]{CJ15}, giving the claim.
\end{proof}

\begin{example}
Let $k$ be a field not algebraic over a finite field and let $\alpha \in k$ be an element of infinite multiplicative order. Consider the ring $A:=k[x_1,x_2,x_3,x_4,x_5]/I_{\alpha}$ where 
\[I_{\alpha}:=(\alpha x_1x_3+x_2x_3,x_1x_4+x_2x_4,x_3^2+\alpha x_1x_5-x_2x_5,x_4^2+x_1x_5-x_2x_5,x_1^2,x_2^2,x_3x_4,x_3x_5,x_4x_5,x_5^2).\]
Then
\begin{enumerate}
\item[$(1)$] $A$ does not satisfy derived $\ldep$.
\item[$(2)$] $A[\![t]\!]$ does not satisfy $\ldep$.
\item[$(3)$] There is a local domain $B$ of dimension $1$ for which $\hat{B} \cong A[\![t]\!]$, and thus $B$ does not satisfy $\ldep$.
\end{enumerate}

\begin{proof}
It follows from \cite{JS04} that $A$ does not satisfy $\uac$, and therefore by Theorem \ref{ldepthm} $A$ does not satisfy derived $\ldep$. It follows from Corollary \ref{ldepcomplete} that $A[\![t]\!]$ does not satisfy derived $\ldep$ and thus does not satisfy $\ldep$ by Theorem \ref{ldepthm}. It follows from \cite[Theorem 1]{Le86} that there is a local domain $B$ for which $\hat{B} \cong A[\![t]\!]$, and $B$ does not satisfy $\ldep$ by Theorem \ref{ldepcomplete}.
\end{proof}
\end{example}

\section{The $\rdep$ and derived $\rdep$ conditions}\label{rdepsection}

The goal of this section is to explore the counterpart to the work from Section \ref{ldepsection} for the $\rdep$ condition. As a beginning note to illustrate the distinction between the $\ldep$ and $\rdep$ conditions, we observe the following:

\begin{example}\label{derivedrdepArtinian}
If $R$ is Artinian, then $R$ satisfies derived $\rdep$ for modules.
\end{example}

\begin{proof}

Suppose $M,N$ are $R$-modules with $q_R(M,N)<\infty$. Since $R$ is Artinian, it follows from Lemma \ref{negativeqr} that we have $\depth_R(M \otimes^L_R N)=-q_R(M,N) \le 0=\depth_R(M)+\depth_R(N)-\depth(R)$, so $R$ satisfies derived $\rdep$. 
\end{proof}

In particular, Example \ref{derivedrdepArtinian} shows that one cannot hope for derived $\rdep$ for modules to behave well under modding out a regular sequence in general, however, as we will ultimately show that derived $\rdep$ for modules is equivalent to derived $\rdep$ for CM rings of positive dimension, the following will show the Artinian case provides the only obstruction:

\begin{prop}\label{cutsdown}
Suppose $R$ is local ring and that $x \in \m$ is a nonzerodivisor. Then $R$ satisfies derived $\rdep$ if and only if $R/x$ satisfies derived $\rdep$.
\end{prop}

\begin{proof}
The proof is analogous to that of Proposition \ref{ldepcutsdown}, mutatis mutandis. First we suppose $R/x$ satisfies derived $\rdep$ and let $M$ and $N$ be $R$-complexes such that $M,N$, and $M \otimes^L_R N$ have bounded homology. From the exact sequences
\[0 \rightarrow M \xrightarrow{\cdot x} M \rightarrow M \otimes^L_R R/x \rightarrow 0\]
and 
\[0 \rightarrow N \xrightarrow{\cdot x} N \rightarrow N \otimes^L_R R/x \rightarrow 0\]
we see that $M \otimes^L_R R/x$ and $N \otimes^L_R R/x$ have bounded homology. Moreover, we have $(M \otimes^L_R R/x) \otimes^L_{R/x} (N \otimes^L_R R/x) \simeq (M \otimes^L_R N) \otimes^L_R R/x$ which also has bounded homology from the short exact sequence
\[0 \rightarrow M \otimes^L_R N \xrightarrow{\cdot x} M \otimes^L_R N \rightarrow (M \otimes^L_R N) \otimes^L_R R/x \rightarrow 0.\]
Since $R/x$ satisfies derived $\rdep$, we have
\[\depth_{R/x}((M \otimes^L_R N) \otimes^L_R R/x)=\depth_{R/x}((M \otimes^L_R R/x) \otimes^L_{R/x} (N \otimes^L_R R/x))\]
\[\le \depth_{R/x}(M \otimes^L_R R/x)+\depth_{R/x}(N \otimes^L_R N)-\depth(R/x)\]
\[=\depth_{R/x}(M \otimes^L_R R/x)+\depth_{R/x}(N \otimes^L_R N)-\depth(R)+1.\]
But then it follows from Lemma \ref{depthcutdown} that
\[\depth_R(M \otimes^L_R N) \le \depth_R(M)+\depth_R(N)-\depth(R),\]
so $R$ satisfies derived $\rdep$.

 For the converse, suppose $R$ satisfies derived $\rdep$, and let $M$ and $N$ be $R/x$-complexes for which $M$, $N$, and $M \otimes^L_{R/x} N$ have bounded homology. From Lemma \ref{ses} we have a short exact sequence
\[(*):0 \to (M \otimes^L_{R/x} N)[-1] \to M \otimes^L_R N \to M \otimes^L_{R/x} N \to 0.\]
It follows the exact sequence (*) that $M \otimes^L_R N$ has bounded homology, and since $R$ satisfies derived $\rdep$, we have 
\[\depth_R(M \otimes^L_R N) \le \depth_R(M)+\depth_R(N)-\depth(R).\]
But applying Proposition \ref{depthlemma} to (*), we see that $\depth_R(M \otimes^L_R N)=\depth_{R/x}(M \otimes^L_{R/x} N)-1$, so
\[\depth_R(M \otimes^L_{R/x} N) \le \depth_{R}(M)+\depth_{R}(N)-\depth(R)+1=\depth_{R/x}(M)+\depth_{R/x}(N)-\depth(R/x),\]
and it follows that $R/x$ satisfies derived $\rdep$.
\end{proof}

As with Proposition \ref{ldepcutsdown}, we note that Proposition \ref{cutsdown} does not require $R$ to be CM, and combining Proposition \ref{cutsdown} with \cite[Lemma 2.10]{KL23} we have as an immediate consequence that derived $\rdep$ behaves well under completion.

\begin{cor}\label{rdepcompletes}
If $R$ is a local ring, then $R$ satisfies derived $\rdep$ if and only if $\hat{R}$ satisfies derived $\rdep$. 
\end{cor}

To prove Theorem \ref{rdepintro}, which is the main theorem of this section, we need some additional preparatory results. For the first of these, we recall that the \emph{nonfree locus} $\NF(M)$ of $M$ is given as 
\[\NF(M):=\{\p \in \Spec(R) \mid M_{\p} \mbox{ is not free over } R_{\p}\} .\]
It is easy to see that $\NF(M)=V(\Ann_R(\Ext^1_R(M,\Omega^1_R(M))))$ and so in particular is a closed subset of $\Spec(R)$. We recall also that an $R$-module $M$ has constant rank $r$ if $M_{\p} \cong R_{\p}^{\oplus r}$ for all $\p \in \Ass(R)$, equivalently, if there is an exact sequence of the form $0 \to R^{\oplus r} \to M \to C \to 0$ where $\dim(C)<\dim(R)$.  

The construction and several claims in the next Lemma are owed to \cite[Proposition 4.2]{Ta09}:

\begin{lem}\label{pushforwardlemma}
Let $M$ be a finitely generated $R$-module and let $x \in \m$ be a nonzerodivisor on $R$. Let $F$ be a free cover of $M$ and consider the pushforward diagram:
\[\begin{tikzcd}[cramped]
	0 & {\Omega^1_R(M)} & F & M & 0 \\
	0 & {\Omega^1_R(M)} & L & M & 0.
	\arrow[from=1-1, to=1-2]
	\arrow[from=1-2, to=1-3]
	\arrow["{\cdot x}"', from=1-2, to=2-2]
	\arrow[from=1-3, to=1-4]
	\arrow[from=1-3, to=2-3]
	\arrow[from=1-4, to=1-5]
	\arrow[equals, from=1-4, to=2-4]
	\arrow[from=2-1, to=2-2]
	\arrow[from=2-2, to=2-3]
	\arrow[from=2-3, to=2-4]
	\arrow[from=2-4, to=2-5]
\end{tikzcd}\]
Then we have the following:
\begin{enumerate}
\item[$(1)$] $\depth_R(L)=\depth_R(M)$.
\item[$(2)$] For any finitely generated $R$-module $N$, $q_R(L,N) \le q_R(M,N)$.
\item[$(3)$] If $x$ is not contained in any minimal prime of $\NF(M)$, then $\dim(\NF(L))<\dim(\NF(M))$.
\item[$(4)$] If $x$ is a nonzerodivisor on $M$ and is not contained in any minimal prime of $\NF(M)$, then $\dim(\NF(\Omega^1_R(M/xM)))<\dim(\NF(M))$.
\item[$(5)$] $\rank(L)=\mu_R(M)$.
\end{enumerate}
\end{lem}

\begin{proof}
For (1), let $t:=\depth_R(M)$. It follows from the depth lemma that $\depth_R(L) \ge t$. Applying $\Hom_R(k,-)$ and consider long exact sequences in $\Ext$, we have the following diagram with exact rows:
\[\begin{tikzcd}[cramped]
	{\Ext^t_R(k,F)} & {\Ext^t_R(k,M)} & {\Ext^{t+1}_R(k,\Omega^1_R(M))} \\
	{\Ext^t_R(k,L)} & {\Ext^t_R(k,M)} & {\Ext^{t+1}_R(k,\Omega^1_R(M))}
	\arrow[from=1-1, to=1-2]
	\arrow[from=1-1, to=2-1]
	\arrow[from=1-2, to=1-3]
	\arrow[equals, from=1-2, to=2-2]
	\arrow["{\cdot x}"', from=1-3, to=2-3]
	\arrow["f", from=2-1, to=2-2]
	\arrow[from=2-2, to=2-3]
\end{tikzcd}\]
Since multiplication by $x$ is the zero map on $\Ext^{t+1}_R(k,\Omega^1_R(M))$, it follows from commutativity of the diagram and exactness of the second row that $f$ is surjective. But $\Ext^t_R(k,M) \ne 0$, and so $\Ext^t_R(k,L) \ne 0$ as well, which gives that $\depth_R(L)=t$.

Claim $(2)$ follows at once from applying $- \otimes_R N$ to the exact sequence $0 \to \Omega^1_R(M) \to L\to M \to 0$ and considering the long exact sequence in $\Tor$.

For claim $(3)$, if $x$ is not contained in any minimal prime $\p$ of $\NF(M)$, then $x$ is a unit upon localizing at $\p$, and then the five lemma forces $F_{\p} \cong L_{\p}$. So no minimal prime of $\NF(M)$ is contained in $\NF(L)$. But if $M_{\p}$ is $R_{\p}$-free, then so is $\Omega^1_R(M)_{\p}$ and the sequence $0 \to \Omega^1_R(M)_{\p} \to L_{\p} \to M_{\p} \to 0$ splits, so $L_{\p}$ is $R_{\p}$-free as well. Thus $\NF(L) \subseteq \NF(M)$, giving the claim.  

For claim $(4)$, if $M_{\p}$ is $R_{\p}$-free, then $M_{\p}/xM_{\p}$ has projective dimension at most $1$ over $R_{\p}$, and since $\Omega^1_R(M/xM)_{\p}$ is an $R_{\p}$-syzygy of $M_{\p}/xM_{\p}$, it follows that it is $R_{\p}$-free. So $\NF(\Omega^1_R(M/xM)) \subseteq \NF(M)$, but if $\p$ is a minimal prime of $\NF(M)$, then since $x \notin \p$, we have $(M/xM)_{\p}=0$ and then $\Omega^1_R(M/xM)_{\p}$ is $R_{\p}$-free, so the claim follows.

Finally, for claim $(5)$, it follows from the snake lemma that there is an exact sequence $0 \to F \to L \to \Omega^1_R(M)/x\Omega^1_R(M) \to 0$. But $\dim(\Omega^1_R(M)/x\Omega^1_R(M))<\dim(R)$ and so has rank $0$. The claim follows from additivity of $\rank$. 
\end{proof}

We recall the following lemma that is well-known to experts. We give a short proof due to lack of a suitable reference:
\begin{lem}\label{constant}
Suppose $R$ satisfies Serre's condition $(S_2)$, e.g. $R$ is CM. If $M$ is an $R$-module that is locally free in codimension $1$, then $M$ has constant rank.
\end{lem}
\begin{proof}
Since $R$ satisfies $(S_2)$, $\Ass(R)=\Min(R)$. Let $\p,\q \in \Ass(R)$. We claim $M_{\p}$ and $M_{\q}$ have the same rank as free modules over $R_{\p}$ and $R_{\q}$ respectively. To see this, since $R$ satisfies $(S_2)$ the Hochster-Huneke graph of $R$ is connected (see \cite{HH94}). This means that there is a sequence of associated primes $\p=\p_1,\dots,\p_r=\q$ for which $\hit(\p_i+\p_{i+1}) \le 1$ for all $i$. Then if $\mathfrak{s}$ is a minimal prime of $\p_i+\p_{i+1}$, then $M_{\mathfrak{s}}$ is already a free $R_{\mathfrak{s}}$-module, so localizing further shows that $\rank_{R_{\p_i}}(M_{\p_i})=\rank_{R_{\p_{i+1}}}(M_{\p_{i+1}})$. As this holds for all $i$, we have in particular that $\rank_{R_{\p}}(M_{\p})=\rank_{R_{\q}}(M_{\q})$, giving the claim.  
\end{proof}

\begin{lem}\label{reducetoconstantrank}
Let $R$ be CM with $\dim(R)=d$. If $\rdep$ holds for any pair of modules that are locally free of constant rank on the punctured spectrum of $R$, then $\rdep$ holds for any pair of modules that are locally free on the punctured spectrum of $R$.
\end{lem}

\begin{proof}
Let $M$ and $N$ be $R$-modules with $q_R(M,N)=0$ that are locally free on the punctured spectrum of $R$. As every Artinian ring trivially satisfies $\rdep$, we may suppose $d \ne 0$. If $d>1$, then it follows from Lemma \ref{constant} that $M$ and $N$ have constant rank already, and there is nothing to prove. So we may suppose that $d=1$. Now, if $M$ and $N$ are both MCM, then $\depth_R(M \otimes_R N) \le \depth_R(M)+\depth_R(N)-d=1$, and we're done, so without loss of generality we may suppose $\depth_R(M)=0$. 

Let $x \in \m$ be a nonzerodivisor on $R$, let $F^M$ be a free cover of $M$, and consider the pushforward diagram:
\[\begin{tikzcd}[cramped]
	0 & {\Omega^1_R(M)} & F^M & M & 0 \\
	0 & {\Omega^1_R(M)} & L_M & M & 0.
	\arrow[from=1-1, to=1-2]
	\arrow[from=1-2, to=1-3]
	\arrow["{\cdot x}"', from=1-2, to=2-2]
	\arrow[from=1-3, to=1-4]
	\arrow[from=1-3, to=2-3]
	\arrow[from=1-4, to=1-5]
	\arrow[equals, from=1-4, to=2-4]
	\arrow[from=2-1, to=2-2]
	\arrow[from=2-2, to=2-3]
	\arrow[from=2-3, to=2-4]
	\arrow[from=2-4, to=2-5]
\end{tikzcd}\]
By Lemma \ref{pushforwardlemma}, $L_M$ has constant rank and we have $\depth_R(L_M)=\depth_R(M)=0$.

We claim that $N$ must be MCM. To see this suppose $\depth_R(N)=0$ and consider the pushforward diagram:
\[\begin{tikzcd}[cramped]
	0 & {\Omega^1_R(N)} & F^N & N & 0 \\
	0 & {\Omega^1_R(N)} & L_N & N & 0
	\arrow[from=1-1, to=1-2]
	\arrow[from=1-2, to=1-3]
	\arrow["{\cdot x}"', from=1-2, to=2-2]
	\arrow[from=1-3, to=1-4]
	\arrow[from=1-3, to=2-3]
	\arrow[from=1-4, to=1-5]
	\arrow[equals, from=1-4, to=2-4]
	\arrow[from=2-1, to=2-2]
	\arrow[from=2-2, to=2-3]
	\arrow[from=2-3, to=2-4]
	\arrow[from=2-4, to=2-5]
\end{tikzcd}\]
so that as above we have $\depth_R(L_N)=\depth_R(N)=0$.

By Lemma \ref{pushforwardlemma}, we have $q_R(L_M,N)=0$ and then that $q(L_M,L_N)=0$. Then we have from hypothesis that $\depth_R(L_M \otimes_R L_N) \le 0+0-1=-1$, which cannot be. So $N$ must be MCM.

Now, if $M \otimes_R N$ is MCM, then as $q_R(M,N)=0$, so is $M \otimes^L_R N$, and then Lemma \ref{depthcutdown} gives that $\depth_R(M \otimes^L_R N/xN)=0$. But then Lemma \ref{negativeqr} forces $q_R(M,N/xN)=0$, and then we have $q_R(L_M,N/xN)=0$ from Lemma \ref{pushforwardlemma} (2). As $L_M$ and $N/xN$ have constant rank, we have from hypothesis that $\depth_R(L_M \otimes_R N/xN) \le \depth_R(L_M)+\depth_R(N/xN)-1=0+0-1$, a contradiction. So $\depth_R(M \otimes_R N)=0 \le \depth_R(M)+\depth_R(N)-d$, completing the proof.
\end{proof}

The next Theorem forms a key part of the main theorem of this section:

\begin{theorem}\label{rdepvsderived}
Let $R$ be CM with $\dim(R)=d$. Suppose $\rdep$ holds for any pair of $R$-modules that are locally free of constant rank on the punctured spectrum of $R$. Then $R$ satisfies derived $\rdep$ for modules. 
\end{theorem}

\begin{proof}
Suppose $M$ and $N$ are $R$-modules with $q:=q_R(M,N)<\infty$. We proceed by induction on $\dim(\NF(M))+\dim(\NF(N))$. For the base case, if $\dim(\NF(M))+\dim(\NF(N))=0$, then both $M$ and $N$ are locally free on the punctured spectrum of $R$. 

If $q=0$, then claim follows from Lemma \ref{reducetoconstantrank}, so we may suppose $q>0$. Then $\Tor^R_q(M,N))=0$ has finite length, so Lemma \ref{negativeqr} gives that $\depth_R(M \otimes^L_R N)=-q$. Letting $F_{\bullet}$ be a minimal free resolution of $M$, we may apply Proposition \ref{depthlemma} to the exact sequences
\[0 \to \Omega^i_R(M) \otimes^L_R N \to F^M_i \otimes^L_R N \to \Omega^{i-1}_R(M) \otimes^L_R N \to 0,\]
to see that $\depth_R(\Omega^q_R(M) \otimes^L_R N)=0$. As $q_R(\Omega^q_R(M),N)=0$, it follows from Lemma \ref{reducetoconstantrank} that 
\[0=\depth_R(\Omega^q_R(M) \otimes_R N) \le \depth_R(\Omega^q_R(M))+\depth_R(N)-d \le \depth_R(M)+q+\depth_R(N)-d.\]
So
\[\depth_R(M \otimes^L_R N)=-q \le \depth_R(M)+\depth_R(N)-d,\]
and the base case is established. 

Now suppose $\dim(\NF(M))+\dim(\NF(N))>0$, in particular without loss of generality that $\dim(\NF(M))>0$. We consider several cases:

(i) Suppose $M$ is MCM and let $\underline{x}$ be a maximal regular sequence. Then as $M/\underline{x}M$ is locally free on the punctured spectrum, we have from the base case that $\depth_R(M/\underline{x}M \otimes^L_R N) \le \depth_R(N)-d$. But then from Lemma \ref{depthcutdown}, we have
\[\depth_R(M \otimes^L_R N)=\depth_R(M/\underline{x}M \otimes^L_R N)+d \le \depth_R(N)=\depth_R(M)+\depth_R(N)-d.\]
So the claim is established when $M$ is MCM.

(ii) Suppose $0<\depth_R(M)<d$. By prime avoidance, we may choose $x \in \m$ to be a nonzerodivisor on $M$ that lies outside all the minimal primes of $\NF(M)$. Then from Lemma \ref{pushforwardlemma}, we have $\dim(\NF(\Omega^1_R(M/xM)))<\dim(\NF(M))$, and we may apply inductive hypothesis to obtain
\[\depth_R(\Omega^1_R(M/xM) \otimes^L_R N) \le \depth_R(\Omega^1_R(M/xM))+\depth_R(N)-d=\depth_R(M)+\depth_R(N)-d.\]
In particular, $\depth_R(\Omega^1_R(M/xM) \otimes^L_R N)<\depth_R(N)$ since $\depth_R(M)<d$. Let $F^{M/xM}$ be a free cover of $M/xM$ over $R$ and note that $\depth_R(F^{M/xM} \otimes^L_R N)=\depth_R(F^{M/xM} \otimes_R N)=\depth_R(N)$. Then applying Proposition \ref{depthlemma} to the exact sequence 
\[0 \to \Omega^1_R(M/xM) \otimes^L_R N \to F^{M/xM} \otimes^L_R N \to M/xM \otimes^L_R N \to 0\]
we get $\depth_R(M/xM \otimes^L_R N)=\depth_R(\Omega^1_R(M/xM) \otimes^L_R N)-1 \le \depth_R(M)+\depth_R(N)-d-1$, so $\depth_R(M \otimes^L_R N) \le \depth_R(M)+\depth_R(N)-d$ by Lemma \ref{depthcutdown}, and we have the claim when $0<\depth_R(M)<d$. 

(iii) Suppose $\depth_R(M)=0$ and that $\depth(\Omega^1_R(M) \otimes^L_R N)<\depth_R(N)$. Since $\depth_R(\Omega^1_R(M))>0$, we may appeal to case (ii) to get that $\depth(\Omega^1_R(M) \otimes^L_R N) \le \depth_R(\Omega^1_R(M))+\depth_R(N)-d=\depth_R(N)-d+1$. Since $\depth_R(\Omega^1_R(M) \otimes^L_R N)<\depth_R(N)$, we may apply Proposition \ref{depthlemma} to the exact sequence
\[0 \to \Omega^1_R(M) \otimes^L_R N \to F^M_0 \otimes^L_R N \to M \otimes^L_R N \to 0\]
to see that 
\[\depth_R(M \otimes^L_R N)=\depth_R(\Omega^1_R(M) \otimes^L_R N)-1 \le \depth_R(N)-d=\depth_R(M)+\depth_R(N)-d.\]

(iv) Finally, suppose $\depth_R(M)=0$ and that $\depth_R(\Omega^1_R(M) \otimes^L_R N) \ge \depth_R(N)$. Let $x$ be a nonzerodivisor on $R$ that is not contained in any minimal prime of $\NF(M)$, and consider the pushforward diagram
\[\begin{tikzcd}[cramped]
	0 & {\Omega^1_R(M)} & F & M & 0 \\
	0 & {\Omega^1_R(M)} & L & M & 0.
	\arrow[from=1-1, to=1-2]
	\arrow[from=1-2, to=1-3]
	\arrow["{\cdot x}"', from=1-2, to=2-2]
	\arrow[from=1-3, to=1-4]
	\arrow[from=1-3, to=2-3]
	\arrow[from=1-4, to=1-5]
	\arrow[equals, from=1-4, to=2-4]
	\arrow[from=2-1, to=2-2]
	\arrow[from=2-2, to=2-3]
	\arrow[from=2-3, to=2-4]
	\arrow[from=2-4, to=2-5]
\end{tikzcd}\]
By Lemma \ref{pushforwardlemma}, we have $\depth_R(L)=0$ and that $\dim(\NF(L))<\dim(\NF(M))$, so we may apply inductive hypothesis to obtain that $\depth_R(L \otimes^L_R N) \le \depth_R(N)-d$. Note that $d>0$ since $\dim(\NF(M))>0$, so in particular $\depth_R(L \otimes^L_R N)<\depth_R(N)$. Then we may apply Proposition \ref{depthlemma} to the exact sequence
\[0 \to \Omega^1_R(M) \otimes^L_R N \to L \otimes^L_R N \to M \otimes^L_R N \to 0\]
to see that 
\[\depth_R(M \otimes^L_R N) \le \depth_R(L \otimes^L_R N) \le \depth_R(N)-d=\depth_R(M)+\depth_R(N)-d,\] completing the proof.
\end{proof}

We are now ready to prove the main theorem of this section:
\begin{theorem}\label{rdepthm}
Let $R$ be a CM ring of dimension $d$. Consider the following conditions:
\begin{enumerate}
\item[$(1)$] $R$ satisfies derived $\rdep$.
\item[$(2)$] $R$ satisfies derived $\rdep$ for modules.
\item[$(3)$] $R$ satisfies $\rdep$.
\item[$(4)$] If $M$ and $N$ are $R$-modules with $q_R(M,N)=0$ that are locally free of constant rank on the punctured spectrum of $R$, then $\codepth_R(M \otimes_R N) \ge \codepth_R(M)+\codepth_R(N)$.
\item[$(5)$] If $M$ and $N$ are $R$-modules with $q_R(M,N)<\infty$, and if $M \otimes_R^L N$ is MCM, then $M$ is MCM.
\item[$(6)$] If $M$ and $N$ are $R$-modules with $N$ MCM, with $q_R(M,N)=0$, and with $M \otimes_R N$ MCM, then $M$ is MCM.
\item[$(7)$] $R$ satisfies $\ubc$.
\item[$(8)$] If $M$ and $N$ are $R$-modules with $N$ MCM and $b_R(M,N)=0$, then $M$ is MCM.
\end{enumerate}

Then $(1) \Rightarrow (2) \iff (3) \iff (4) \Rightarrow (5) \Rightarrow (6)$, and $(3) \Rightarrow (7) \Rightarrow (8)$. If $d>0$, then $(2) \Rightarrow (1)$, and if $R$ has a canonical module $\w_R$, then $(6) \iff (8)$.
\end{theorem}

\begin{proof}

$(1) \Rightarrow (2) \Rightarrow (3) \Rightarrow (4)$ are immediate while $(4) \Rightarrow (2)$ is the content of Theorem \ref{rdepvsderived}, so $(2)-(4)$ are equivalent. That $(2) \Rightarrow (1)$ when $d>0$ is the content of Theorem \ref{derivedldepvsformodules}.

It is then immediate that $(4) \Rightarrow (5) \Rightarrow (6)$, so we focus our attention first on $(3) \Rightarrow (7)$.

The claim is automatic if $d=0$, so we may suppose $d>0$, and since we have seen that $(3) \Rightarrow (1)$ in this setting, it suffices to show $(1) \Rightarrow (7)$. Moreover, we have from Corollary \ref{rdepcompletes} that $(1)$ ascends to the completion, while $\ubc$ clearly descends from the completion, so we may suppose $R$ is compete and in particular that $R$ admits a canonical module $\w_R$. Suppose $M$ and $N$ are $R$-modules with $b_R(M,N)<\infty$. Then we have 
\[\RHom_R(M,N) \cong \RHom_R(M,\RHom_R(\RHom_R(N,\w_R),\w_R)) \cong \RHom_R(M \otimes^L_R \RHom_R(N,\w_R),\w_R),\]
and in particular, we see that $\RHom_R(N,\w_R)$ and $M \otimes^L_R \RHom_R(N,\w_R)$ have bounded homology. Since we assume $R$ satisfies derived $\rdep$, we have that $\depth_R(M \otimes^L_R \RHom_R(N,\w_R)) \le \depth_R(M)+\depth_R(\RHom_R(N,\w_R))-d$.

But by local duality (\cite[3.4.1]{IM21}), we have $\depth_R(\RHom_R(N,\w_R))=d-\sup\{i \mid H^i(N) \ne 0\}=d$, so $\depth_R(M \otimes^L_R \RHom_R(N,\w_R)) \le \depth_R(M)$. But then applying local duality again, we see that 
$b_R(M,N)=\sup\{H^i(\RHom_R(M,N)) \ne 0\}=d-\depth_R(M \otimes^L_R \RHom_R(N,\w_R)) \ge \codepth_R(M)$, so $R$ satisfies $\ubc$.

For $(7) \Rightarrow (8)$, if $b_R(M,N)=0$ and $N$ is MCM, the $\ubc$ condition forces $\codepth_R(M)=0$, so $M$ is MCM. 

We now show $(6) \Rightarrow (8)$ assuming $R$ admits a canonical module. Suppose $M$ and $N$ are $R$-modules with $N$ MCM and with $b_R(M,N)=0$. Then by Proposition \ref{exttorall} we have that $M \otimes_R N^{\vee}$ is MCM and that $q_R(M,N^{\vee})=0$, so the condition of $(6)$ forces $M$ to be MCM.

Finally, we show $(8) \Rightarrow (6)$ assuming $R$ admits a canonical module. Suppose $M$ and $N$ are $R$-modules with $M \otimes_R N$ and $N$ MCM and with $q_R(M,N)=0$. Then by Proposition \ref{exttorall}, we have that $b_R(M,N^{\vee})=0$, and then $M$ is MCM by the condition of $(8)$.   
\end{proof}

Combining Theorems \ref{ldepthm} and \ref{rdepthm} yields the following immediate consequence:
\begin{cor}\label{depcor}
Suppose $R$ is CM with $\dim(R)>0$. If $R$ satisfies $\dep$, then for any $R$-modules $M$ and $N$ with $b_R(M,N)<\infty$, we have $b_R(M,N)=\codepth_R(M)$. The converse holds if $R$ is Gorenstein. 
\end{cor}

\begin{proof}
If $R$ is CM with $\dim(R)>0$ and $R$ satisfies $\dep$, then $R$ satisfies condition $(5)$ of Theorem \ref{ldepthm} and satisfies $\ubc$ from Theorem \ref{rdepthm} which combine to give the first claim. For the converse when $R$ is Gorenstein, then condition that $b_R(M,N)=\codepth_R(M)$ when $b_R(M,N)<\infty$ in particular implies condition $(5)$ of Theorem \ref{ldepthm} which thus implies that $R$ satisfies derived $\ldep$. That $R$ satisfies $\dep$ then follows from Corollary \ref{gorcase}.
\end{proof}

\begin{remark}\label{significance}
Theorem \ref{rdepthm} combined with Proposition \ref{cutsdown} gives the precise means for one to check the $\rdep$ condition after modding out a maximal regular sequence $\underline{x}$ in a CM local ring $R$ of positive dimension; namely $R$ will satisfy $\rdep$ if and only if $R/\underline{x}$ satisfies derived $\rdep$. This gives a counterpoint to what happens for $\ldep$, where one may check whether $R/\underline{x}$ satisfies derived $\ldep$ for modules or whether it satisfies $\uac$. In fact of the $8$ conditions discussed in Theorem \ref{rdepthm}, $(1)$ is the only one that does not hold for every Artinian algebra.
\end{remark}

\begin{remark}\label{obstructions}
As one may expect, the $\uac$ and $\ubc$ admit analogous versions for complexes. For an $R$-complex $X$, setting $\sup X:=\sup\{i \mid H^i(X) \ne 0\}$, the definitions are as follows: 
\begin{enumerate}
    \item[$(1)$] $R$ satisfies derived $\uac$ with Auslander bound $b_R$ if for all $R$-complexes $M$,$N$ for which $M$, $N$, and $\RHom_R(M,N)$ have bounded homology, we have $\sup \RHom_R(M,N) \le b_R+\sup N$.
    \item[$(2)$] $R$ satisfies derived $\ubc$ if for all $R$-complexes $M$, $N$ for which $M$, $N$ and $\RHom_R(M,N)$ have bounded homology, we have $\sup \RHom_R(M,N) \ge \codepth_R(M)+\sup N$.
\end{enumerate}
Following the approach of Corollaries \ref{ldepcompletes} and \ref{rdepcompletes}, one can see that derived $\uac$ and derive $\ubc$ ascend and descend to and from the completion, and then a straightforward application of local duality \cite[3.4.1]{IM21} shows that derived $\ldep$ is equivalent to derived $\uac$ with $b_R=\depth(R)$, and that derived $\rdep$ is equivalent derived $\ubc$. It then follows from Theorem \ref{ldepthm} that derived $\uac$ is equivalent to $\uac$. However, we do not know whether $\ubc$ implies derived $\ubc$ when $\depth(R)>0$, or even for CM rings of positive dimension. The chief obstruction in the latter case is that we don't know whether $\rdep$ can be checked on MCM modules in the same vein as $\ldep$. If this can be shown, then it will follow that derived $\ubc$ and $\ubc$ are equivalent when $R$ is CM of positive dimension, and the Gorenstein assumption for the converse in Corollary \ref{depcor} can be removed.
\end{remark}

While we can construct examples that fail $\rdep$, the only means we know of doing so works through the following which extends \cite[Proposition 3.5]{KL23}:
\begin{cor}\label{rdepimpliestr}
Suppose $R$ is CM with $\dim(R)>0$. If $R$ satisfies $\rdep$ or if $R$ satisfies $\ubc$, then $R$ satisfies $\tr$.
\end{cor}

\begin{proof}
If $R$ satisfies either $\rdep$ or $\ubc$, then by Theorem \ref{rdepthm}, if $\Ext^{i>0}_R(M,R)=0$, it must be that $M$ is MCM. In particular, $\depth_R(M)=\depth(R)>0$, and then $R$ satisfies $\tr$ by \cite[Theorem 3.4]{KL23}. 
\end{proof}

\begin{example}\label{trcountex}
Let $k$ be a field not algebraic over a finite field and let $\alpha \in k$ be an element of infinite multiplicative order. Let $R=k[\![v,x,y,z]\!]/I$ where $I=(v^2,z^2,xy,vx+\alpha xz,vy+yz,vx+y^2,vy-x^2)$ and let $A:=R[\![t]\!]$. Then 
\begin{enumerate}
\item[$(1)$] $R$ satisfies derived $\rdep$, but does not satisfy derived $\ldep$. 
\item[$(2)$] $A$ does not satisfy $\rdep$ or $\ldep$.
\item[$(3)$] There is a local domain $B$ of dimension $1$ with $\hat{B} \cong A$ and $B$ does not satisfy $\rdep$ or $\ldep$. 
\end{enumerate}
\end{example}

\begin{proof}
As $R$ is Artinian, it satisfies derived $\rdep$ from Example \ref{derivedrdepArtinian}. It follows from \cite{JS06} that $R$ does not satisfy $\tr$, and then neither does $A$ by \cite[Theorem 2.11]{KL23}. In particular, neither $R$ nor $A$ satisfies derived $\ldep$ by Corollary \ref{ldepimpliestr}. Then by Theorem \ref{ldepthm}, $A$ does not satisfy $\ldep$. By \cite[Theorem 1]{Le86}, there is a local domain $B$ for which $\hat{B} \cong A$. Since $A$ does not satisfy $\tr$, neither does $B$ by \cite[Theorem 2.11]{KL23}, and then as above, $B$ satisfies neither $\rdep$ nor $\ldep$.
\end{proof}

\begin{remark}
The construction of Example \ref{trcountex} that provides examples failing the $\tr$ condition is the only method we know for producing examples that fail to satisfy $\rdep$ and it has some limitations e.g.~since these examples have no hope of satisfying $\ldep$. 
\end{remark}

The next example is inspired by a construction of \cite{NSW19} and shows that none of the conditions $\ldep$, $\rdep$, $\ubc$, and $\tr$ localize in general:
\begin{example}\label{notlocalize}
Let $R$, $I$, and $A$ be as in Example \ref{trcountex}. Let $B:=k[\![t,v,x,y,z,s]\!]/(I,ts,vs,xs,ys,zs) \cong A \times_k k[\![s]\!]$ where $\times_k$ denotes the fiber product, take $Q:=B[\![w]\!]$, and let $\p:=(w,v,x,y,z,s)$. Then $Q$ is a complete Cohen-Macaulay local ring of dimension $2$ by \cite[Facts 2.1 and 2.2]{NSW19}. It follows from \cite[Corollary 6.8]{NT20} and \cite[Theorem 3.2 (2)]{LM20} that $Q$ satisfies the trivial vanishing condition of \cite{LM20} and thus satisfies $\dep$, $\ubc$, and $\tr$ by Example \ref{tvexample}. However, we observe that $Q_{\p} \cong k(t)[\![w,v,x,y,z]\!]/(v^2,z^2,xy,vx+\alpha xz,vy+yz,vx+y^2,vy-x^2)$ which does not satisfy $\tr$ by \cite[Proposition 2.6 (2)]{KL23} and \cite{JS06}, and then also fails to satisfy both $\ldep$, $\rdep$, and $\ubc$ by Corollaries \ref{ldepimpliestr} and \ref{rdepimpliestr}. 
\end{example}

\section{Calculation of $q_R(M,N)$}\label{qrsection}
As the work of the previous sections indicates, the value $q_R(M,N)$, when it is finite, is intimately connected to the depth behavior of $M \otimes_R N$ is relation to that of $M$ and $N$. In this section, we seek to make this connection more precise, ultimately extending work of Jorgensen \cite{Jo99}. Unlike previous sections, we will not require $R$ to be CM in this section. As noted in Example \ref{notlocalize}, the $\ldep$ and $\rdep$ conditions need not localize, so assuming these conditions hold locally on $\Spec(R)$ will give meaningful hypotheses.

\begin{lem}\label{lemmaoneforjorgensensformula} Let $R$ be a Noetherian local ring and let $M$, $N$ be finitely generated modules. Suppose derived $\rdep$ holds on $\Spec{R}$ for the module $M$. If $q_{R}(M,N)<\infty$, then
\[q_{R}(M,N)\geq\sup\{\depth(R_{\p})-\depth_{R_{\p}}(M_{\p})-\depth_{R_{\p}}(N_{\p}) \mid \p\in\Supp{M}\cap\Supp{N}\}.\]
\end{lem}

\begin{proof}
Note since $q_R(M,N)<\infty$ that $q_{R_{\p}}(M_{\p},N_{\p})<\infty$ for all $\p \in \Spec(R)$. Then by assumption we have \[\depth_{R_{\p}}(M_{\p} \otimes_{R_{p}}^{L} N_{\p})\leq\depth_{R_{\p}}(M_{\p})+\depth_{R_{\p}}(N_{\p})-\depth(R_{\p})\]
for any $\p \in \Supp(M) \cap \Supp(N)$. Therefore
Lemma \ref{negativeqr} gives that \[q_{R}(M,N) \ge q_{R_{\p}}(M_{\p},N_{\p}) \geq -\depth_{R_{\p}}(M_{\p} \otimes_{R_{\p}}^{L} N_{\p}) \geq \depth(R_{\p}) - \depth_{R_{\p}}(M_{\p}) - \depth_{R_{\p}}(N_{\p})\]
for all $\p \in \Supp(M) \cap \Supp(N)$, as desired.
\end{proof}

\begin{lem}\label{lemmatwoforjorgensensformula} Let $R$ be a Noetherian local ring and let $M$, $N$ be finitely generated modules. Suppose derived $\ldep$ holds on $\Spec{R}$ for the module $M$. If $q_{R}(M,N)<\infty$, then
\[q_{R}(M,N)\leq\sup\{\depth{R_{\p}}-\depth{M_{\p}}-\depth{N_{\p}} \mid \p \mbox{ is a minimal prime of } \Tor_{q_{R}(M,N)}^{R}(M,N)\}.\]
In particular, we have
\[q_{R}(M,N) \le \sup\{\depth(R_{\p})-\depth_{R_{\p}}(M_{\p})-\depth_{R_{\p}}(N_{\p}) \mid \p\in\Supp{M}\cap\Supp{N}\}.\]
\end{lem}

\begin{proof}
By assumption, we have
\[\depth_{R_{\p}}(M_{\p} \otimes_{R_{\p}}^{L} N_{\p})\geq\depth_{R_{\p}}(M_{\p})+\depth_{R_{\p}}(N_{\p})-\depth(R_{\p})\]
for all $\p \in \Spec(R)$. 
Then if $\p$ is a minimal prime of $\Tor^R_{q_R(M,N)}(M,N)=0$, we have $q_R(M,N)=q_{R_{\p}}(M_{\p},N_{\p})$ and since $\depth_{R_{\p}}(\Tor^{R_{\p}}_{q_{R_{\p}}(M_{\p},N_{\p})}(M_{\p},N_{\p}))=0$, Lemma \ref{negativeqr} gives that
\[q_{R}(M,N)= q_{R_{\p}}(M_{\p},N_{\p}) = -\depth_{R_{\p}}(M_{\p} \otimes_{R_{\p}}^{L} N_{\p}) \leq \depth(R_{\p}) - \depth_{R_{\p}}(M_{\p}) - \depth_{R_{\p}}(N_{\p}),\]
and the claim follows.
\end{proof}

The following is an immediate consequence of Lemma \ref{lemmaoneforjorgensensformula} and Lemma \ref{lemmatwoforjorgensensformula}. 
\begin{theorem}\label{qrformula} Let $R$ be a Noetherian local ring and let $M$, $N$ be finitely generated modules. Assume derived $\dep$ holds on $\Spec{R}$ for the module $M$. If $q_{R}(M,N)<\infty$, then
\[q_{R}(M,N)=\sup\{\depth(R_{\p})-\depth_{R_{\p}}(M_{\p})-\depth_{R_{\p}}(N_{\p}) \mid \p\in\Supp{M}\cap\Supp{N}\}.\]
\end{theorem}

\begin{remark}\label{jorrmk} In \cite[Theorem 2.2]{Jo99}, Jorgensen showed the formula
\[q_{R}(M,N)=\sup\{\depth(R_{\p})-\depth_{R_{\p}}(M_{\p})-\depth_{R_{\p}}(N_{\p}) \mid \p\in\Supp{M}\cap\Supp{N}\}.\]
holds whenever $q_R(M,N)<\infty$ and $M$ has finite complete intersection dimension (see e.g. \cite{AG97} as a reference for this theory). Work of Iyengar shows that derived $\dep$ holds for $M$ whenever $M$ has finite complete intersection dimension \cite[Theorem 4.3]{iyengar}. As finiteness of complete intersection dimension is known to localize (see \cite[Proposition 1.6]{AG97}), Theorem \ref{qrformula} thus gives a direct extension of \cite[Theorem 2.2]{Jo99}.
\end{remark}

\section{Questions}\label{questionsection}

The work of the proceeding sections leave open several natural questions. 
Given the obstructions noted in Remark \ref{obstructions} it is natural to ask:
\begin{quest}
Do any of the remaining implications hold in Theorem \ref{rdepthm}? Of special interest, can $\rdep$ be checked on pairs $M,N$ for which $N$ is MCM?
\end{quest}
As noted in Remark \ref{obstructions}, if one can show $\rdep$ can be checked on pairs $M,N$ where $N$ is MCM, then it will follow as a consequence that $\ubc$ implies $\rdep$.

Currently, the only approach we have to construct examples that fail $\rdep$ is to build examples that fail the $\tr$ condition and will thus fail each condition in Theorem \ref{rdepthm}. If the remaining implication of Theorem \ref{rdepthm} end up failing to hold, it would require more exotic constructions to produce counterexamples.

Since Corollary \ref{gorcase} shows that $\ldep$ implies $\rdep$ when $R$ is Gorenstein, it is natural to pose the following:
\begin{quest}
Suppose $R$ is CM.
\begin{enumerate}
\item[$(1)$] Does $\ldep$ always imply $\rdep$?
\item[$(2)$] Does $\rdep$ always imply $\ldep$? What if $R$ is Gorenstein?
\end{enumerate}
\end{quest}
In light of Example \ref{trcountex} and the work of \cite{KO22}, one can find domains of dimension $1$ or normal CM isolated singularities in higher dimension which fail $\tr$, and thus fail $\ldep$ and $\rdep$. However, such examples depend on constructions of Lech and Heitmann, and as a consequence are not excellent \cite{Le86,He94}. It is difficult to see any direct implication of excellence on the $\ldep$, $\rdep$, and $\tr$ conditions, but the graded setting offers notable benefits, namely the implications of $\Ext/\Tor$ vanishing on Hilbert functions. We thus ask the following: 

\begin{quest}
\
\begin{enumerate}
\item[$(1)$] Is there a positively graded generically Gorenstein $k$-algebra which does not satisfy $\ldep$ or $\rdep$? 
\item[$(2)$] Is there a numerical semigroup ring which does not satisfy $\ldep$ or $\rdep$?
\item[$(3)$] Is there a numerical semigroup ring which does not satisfy $\tr$?
\end{enumerate}
\end{quest}

In Theorem \ref{ldepthm}, the equivalence of condition $(5)$ requires the existence of a canonical module, and the behavior is not so clear without this hypothesis. We thus ask:
\begin{quest}
Does condition $(5)$ of Theorem \ref{ldepthm} ascend to the completion?
\end{quest}

Finally, the work of Sections \ref{ldepsection} and \ref{rdepsection} depends fundamentally on the CM assumption, and so we ask:

\begin{quest}
What is the situation when $R$ is not CM? Does derived $\ldep$ still characterize $\uac$ with $b_R=\depth(R)$ in this setting? That derived $\ldep$ implies $\uac$ follows as noted in Remark \ref{obstructions}, but what about the converse?
\end{quest}

\section*{Acknowledgements}

This material is based upon work supported by the National Science Foundation under Grant No. DMS-1928930 and by the Alfred P. Sloan Foundation under grant G-2021-16778, while Lyle was in residence at the Simons Laufer Mathematical Sciences Institute (formerly MSRI) in Berkeley, California, during the Spring 2024 semester. Kimura was party supported by Grant-in-Aid for JSPS Fellows Grant Number 23KJ1117. We are grateful to Yuya Otake and Ryo Takahashi for helpful discussions on the $\ldep$ and $\rdep$ conditions and for giving useful feedback on an earlier draft of the paper. We also thank Hiroki Matsui for helpful discussions.

\bibliographystyle{alpha}
\bibliography{mybib}

\begin{thebibliography}{NSWTV19}

\bibitem[AGP97]{AG97}
Luchezar~L. Avramov, Vesselin~N. Gasharov, and Irena~V. Peeva.
\newblock Complete intersection dimension.
\newblock {\em Inst. Hautes \'Etudes Sci. Publ. Math.}, (86):67--114, 1997.

\bibitem[AINSW22]{AI22}
Luchezar~L. Avramov, Srikanth~B. Iyengar, Saeed Nasseh, and Keri Sather-Wagstaff.
\newblock Persistence of homology over commutative noetherian rings.
\newblock {\em J. Algebra}, 610:463--490, 2022.

\bibitem[Aus61]{Au61}
M.~Auslander.
\newblock Modules over unramified regular local rings.
\newblock {\em Illinois J. Math.}, 5:631--647, 1961.

\bibitem[BH93]{BH93}
Winfried Bruns and J\"{u}rgen Herzog.
\newblock {\em Cohen-{M}acaulay rings}, volume~39 of {\em Cambridge Studies in Advanced Mathematics}.
\newblock Cambridge University Press, Cambridge, 1993.

\bibitem[BJM20]{BJ20}
Petter~A. Bergh, David~A. Jorgensen, and W.~Frank Moore.
\newblock A converse to a construction of {E}isenbud-{S}hamash.
\newblock {\em J. Commut. Algebra}, 12(4):467--477, 2020.

\bibitem[CH10]{CH10}
Lars~Winther Christensen and Henrik Holm.
\newblock Algebras that satisfy {A}uslander's condition on vanishing of cohomology.
\newblock {\em Math. Z.}, 265(1):21--40, 2010.

\bibitem[CJ15]{CJ15}
Lars~Winther Christensen and David~A. Jorgensen.
\newblock Vanishing of {T}ate homology and depth formulas over local rings.
\newblock {\em J. Pure Appl. Algebra}, 219(3):464--481, 2015.

\bibitem[DEL21]{DE21}
Hailong Dao, Mohammad Eghbali, and Justin Lyle.
\newblock Hom and {E}xt, revisited.
\newblock {\em J. Algebra}, 571:75--93, 2021.

\bibitem[DT15]{DT15}
Hailong Dao and Ryo Takahashi.
\newblock Classification of resolving subcategories and grade consistent functions.
\newblock {\em International Mathematics Research Notices}, 2015(1):119--149, 2015.

\bibitem[FI03]{FI03}
Hans-Bj\o~rn Foxby and Srikanth Iyengar.
\newblock Depth and amplitude for unbounded complexes.
\newblock In {\em Commutative algebra ({G}renoble/{L}yon, 2001)}, volume 331 of {\em Contemp. Math.}, pages 119--137. Amer. Math. Soc., Providence, RI, 2003.

\bibitem[Hei94]{He94}
Raymond~C. Heitmann.
\newblock Completions of local rings with an isolated singularity.
\newblock {\em J. Algebra}, 163(2):538--567, 1994.

\bibitem[HH94]{HH94}
Melvin Hochster and Craig Huneke.
\newblock Indecomposable canonical modules and connectedness.
\newblock In {\em Commutative algebra: syzygies, multiplicities, and birational algebra ({S}outh {H}adley, {MA}, 1992)}, volume 159 of {\em Contemp. Math.}, pages 197--208. Amer. Math. Soc., Providence, RI, 1994.

\bibitem[HW94]{HW94}
Craig Huneke and Roger Wiegand.
\newblock Tensor products of modules and the rigidity of {${\rm Tor}$}.
\newblock {\em Math. Ann.}, 299(3):449--476, 1994.

\bibitem[IMSW21]{IM21}
Srikanth~B. Iyengar, Linquan Ma, Karl Schwede, and Mark~E. Walker.
\newblock Maximal {C}ohen-{M}acaulay complexes and their uses: a partial survey.
\newblock In {\em Commutative algebra}, pages 475--500. Springer, Cham, [2021] \copyright 2021.

\bibitem[Iye99]{iyengar}
S.~Iyengar.
\newblock Depth for complexes, and intersection theorems.
\newblock {\em Math. Z.}, 230(3):545--567, 1999.

\bibitem[JcS04]{JS04}
David~A. Jorgensen and Liana~M. \c~Sega.
\newblock Nonvanishing cohomology and classes of {G}orenstein rings.
\newblock {\em Adv. Math.}, 188(2):470--490, 2004.

\bibitem[JcS06]{JS06}
David~A. Jorgensen and Liana~M. \c~Sega.
\newblock Independence of the total reflexivity conditions for modules.
\newblock {\em Algebr. Represent. Theory}, 9(2):217--226, 2006.

\bibitem[Jor99]{Jo99}
David~A. Jorgensen.
\newblock A generalization of the {A}uslander-{B}uchsbaum formula.
\newblock {\em J. Pure Appl. Algebra}, 144(2):145--155, 1999.

\bibitem[KLOT23]{KL23}
Kaito {Kimura}, Justin {Lyle}, Yuya {Otake}, and Ryo {Takahashi}.
\newblock {On the vanishing of Ext modules over a local unique factorization domain with an isolated singularity}.
\newblock {\em arXiv e-prints}, page arXiv:2310.16599, October 2023.

\bibitem[KOT22]{KO22}
Kaito Kimura, Yuya Otake, and Ryo Takahashi.
\newblock Maximal cohen--macaulay tensor products and vanishing of ext modules.
\newblock {\em Bulletin of the London Mathematical Society}, 54(6):2456--2468, 2022.

\bibitem[Lec86]{Le86}
Christer Lech.
\newblock A method for constructing bad {N}oetherian local rings.
\newblock In {\em Algebra, algebraic topology and their interactions ({S}tockholm, 1983)}, volume 1183 of {\em Lecture Notes in Math.}, pages 241--247. Springer, Berlin, 1986.

\bibitem[LMn20]{LM20}
Justin Lyle and Jonathan Monta\~{n}o.
\newblock Extremal growth of {B}etti numbers and trivial vanishing of (co)homology.
\newblock {\em Trans. Amer. Math. Soc.}, 373(11):7937--7958, 2020.

\bibitem[LW12]{LW12}
Graham~J. Leuschke and Roger Wiegand.
\newblock {\em Cohen-{M}acaulay representations}, volume 181 of {\em Mathematical Surveys and Monographs}.
\newblock American Mathematical Society, Providence, RI, 2012.

\bibitem[NSWTV19]{NSW19}
Saeed Nasseh, Sean Sather-Wagstaff, Ryo Takahashi, and Keller VandeBogert.
\newblock Applications and homological properties of local rings with decomposable maximal ideals.
\newblock {\em J. Pure Appl. Algebra}, 223(3):1272--1287, 2019.

\bibitem[NT20]{NT20}
Saeed Nasseh and Ryo Takahashi.
\newblock Local rings with quasi-decomposable maximal ideal.
\newblock {\em Math. Proc. Cambridge Philos. Soc.}, 168(2):305--322, 2020.

\bibitem[Tak09]{Ta09}
Ryo Takahashi.
\newblock Modules in resolving subcategories which are free on the punctured spectrum.
\newblock {\em Pacific J. Math.}, 241(2):347--367, 2009.

\bibitem[Yos05]{Yo05}
Yuji Yoshino.
\newblock A functorial approach to modules of {G}-dimension zero.
\newblock {\em Illinois J. Math.}, 49(2):345--367, 2005.

\end{thebibliography}

\vspace{.3cm}

\end{document}